\documentclass[11pt]{amsart}
\usepackage{amsmath,amsthm, amscd, amssymb, amsfonts, mathrsfs,pb-diagram,pb-xy}
\usepackage{amsxtra, color}
\usepackage{enumitem}
\usepackage[all]{xy}
\usepackage{amssymb}

\usepackage{setspace}
\usepackage{hyperref}

\usepackage{mathabx}

\usepackage[T2A,T1]{fontenc}

\usepackage{tikz}
\usetikzlibrary{patterns}
\usetikzlibrary{arrows}
\usetikzlibrary{shapes.misc, positioning}
\usetikzlibrary{backgrounds}
\usetikzlibrary{shapes.multipart}
\usetikzlibrary{decorations.pathmorphing}
\usetikzlibrary{positioning}

\usetikzlibrary{calc}
\tikzset{make origin horizontal center of bounding box/.style={%
execute at end picture={%
\path let \p1=(current bounding box.west),\p2=(current bounding box.east)
in ({-max(-1*\x1,\x2)},\y1) ({max(-1*\x1,\x2)},\y1);
}}}

\pgfdeclarepatternformonly{my crosshatch dots}{\pgfqpoint{-1pt}{-1pt}}{\pgfqpoint{5pt}{5pt}}{\pgfqpoint{6pt}{6pt}}%
{
    \pgfpathcircle{\pgfqpoint{0pt}{0pt}}{.5pt}
    \pgfpathcircle{\pgfqpoint{3pt}{3pt}}{.5pt}
    \pgfusepath{fill}
}

\usepackage{chngcntr}

\counterwithin*{equation}{section}

\newcommand\m[1]{\boldsymbol{|}#1\boldsymbol{\rangle}}

\newcommand{\I}{\mathbb{I}}
\newcommand{\N}{\mathbb{N}}

\newcommand{\fp}{\mathfrak{p}}

\newcommand{\fq}{\mathfrak{q}}
\newcommand{\fr}{\mathfrak{r}}

\newcommand{\fz}{\mathfrak{z}}

\newcommand{\soc}{\operatorname{soc}}

\newcommand{\Inf}{\operatorname{Inf}}

\newcommand{\odd}{\operatorname{odd}}
\newcommand{\car}{\operatorname{car}}
\newcommand{\aff}{\operatorname{aff}}
\newcommand{\link}{\operatorname{link}}

\newcommand{\ot}{\otimes}

\newcommand{\BV}{\mathfrak{B}}

\theoremstyle{plain}

\newtheorem{theorem}{Theorem}[section]
\newtheorem{lemma}[theorem]{Lemma}
\newtheorem{definition-theorem}[theorem]{Definition-Theorem}
\newtheorem{proposition}[theorem]{Proposition}

\newtheorem{corollary}[theorem]{Corollary}

\theoremstyle{definition}
\newtheorem{definition}[theorem]{Definition}
\newtheorem{example}[theorem]{Example}

\newtheorem{remark}[theorem]{Remark}

\newcommand{\Z}{\mathbb Z}

\newcommand{\ku}{\Bbbk}

\newcommand{\cG}{\mathcal{G}} 
\newcommand{\cC}{\mathcal{C}} 
\newcommand{\cD}{\mathcal{D}} 
\newcommand{\CC}{\mathbb C}

\newcommand{\cJ} {{\mathcal{J}}}

\newcommand{\cR} {{\mathcal{R}}}

\newcommand\cM{\mathcal{M}}
\newcommand\cW{\mathcal{W}}
\newcommand{\cX}{\mathcal{X}}

\newcommand{\cH} {{\mathcal{H}}}

\newcommand{\bk} {\mathbf{k}}

\newcommand{\bA} {\mathbf{A}}

\newcommand \id { {\mathrm{id}} }

\newcommand \rank { {\mathrm{rank}} }

\newcommand \ch   { {\mathrm{ch}} }

\newcommand{\Aut}{\operatorname{Aut}}
\newcommand{\Alg}{\operatorname{Alg}}

\DeclareMathOperator \ord { {\mathrm{ord}} }
\newcommand\ad{\operatorname{ad}}

\DeclareMathOperator \Ker { {\mathrm{Ker}} }

\newcommand{\Hom}{\operatorname{Hom}}

\DeclareMathOperator \Ext { {\mathrm{Ext}} }

\renewcommand \Im { {\mathrm{Im}} }

\newcommand\fip{f_i}
\newcommand\gip{g_i}

\begin{document}

\title[Linkage principle for  small quantum groups]{Linkage principle for  small quantum groups}

\author[Vay]{Cristian Vay}

\address{\newline\noindent Facultad de Matem\'atica, Astronom\'\i a, F\'\i sica y Computaci\'on,
Universidad Nacional de C\'ordoba. CIEM--CONICET. Medina Allende s/n, 
Ciudad Universitaria 5000 C\'ordoba, Argentina}
\email{cristian.vay@unc.edu.ar}


\thanks{This work is partially supported by  CONICET PIP 11220200102916CO, Foncyt PICT 2020-SERIEA-02847 and Secyt (UNC)}

\begin{abstract}

We consider small quantum groups with root systems of Cartan, super and modular type, among others. These are constructed as Drinfeld doubles of finite-dimensional Nichols algebras of diagonal type. We prove a linkage principle for them by adapting techniques from the work of Andersen, Jantzen and Soergel in the context of small quantum groups at roots of unity. Consequently we characterize the blocks of the category of modules. We also find a notion of (a)typicality similar to the one in the representation theory of Lie superalgebras. The typical simple modules turn out to be the simple and projective Verma modules. Moreover, we deduce a character formula for $1$-atypical simple modules.

\end{abstract}

\maketitle

\section{Introduction}

The small quantum groups at roots of unity introduced by Lusztig are beautiful examples of finite-dimensional Hopf algebras with relevant applications in several subjects. They are instrumental in Lusztig's program concerning the representation theory of algebraic groups in positive characteristic \cite{L-ModRepThe,L-qgps-at-roots,L-findim,L-onqg}. In this framework, he conjectured a character formula for their simple modules, and that this formula implies an analogous one for restricted Lie algebras in positive characteristic. The validity of the ``quantum formula'' was confirmed by the works of Kazhdan--Lusztig \cite{KL1,KL2} and Kashiwara--Tanisaki \cite{KT}. Andersen, Jantzen and Soergel \cite{AJS} then proved that the ``quantum formula'' implies the ``modular formula'' for large enough characteristics (although, we have to mention that 20 years later Williamson \cite{williamson} showed that this ``enough'' could be very big). More recently, Fiebig \cite{fiebig} gave a new proof for the ``quantum formula'' building up on certain categories defined in \cite{AJS} and, joint to Lanini, they have carried out a further development on these categories, cf. \cite{FL}.

\begin{figure}[h]
\begin{tikzpicture}

\node[label=above:{$u_q(\mathfrak{g})$},align=center] (uq) at (0,0) {Small quantum group \\ at a root of unity};

\node[label=above:{$U^{[p]}(\mathfrak{g}_k)$},align=center] (uq) at (5,0) {Restricted \\ enveloping algebra};

\end{tikzpicture}
\caption{Algebras involved in Lusztig's conjectures. Given a finite root system $\Delta$, we write $\mathfrak{g}$ and $\mathfrak{g}_k$ for the associated Lie algebras over $\mathbb{C}$ and over 
an algebraically closed field $k$ of characteristic $p$ odd ($\neq3$ if $\mathfrak{g}$ has a component of type $G_2$). We write $q$ for a $p$-th root of unity in $\mathbb{C}$.}
\label{fig:uqg}
\end{figure}

Here, we adapt tools and techniques from \cite{AJS} to investigate the representation theory of a more general class of small quantum groups:  the Drinfeld doubles of bosonizations of finite-dimensional Nichols algebras of diagonal type over abelian groups, see Figure \ref{fig:uq}. The interest in Nichols algebras have been continually increasing since Andruskiewitsch and Schneider assigned them a central role in the classification of pointed Hopf algebras \cite{AS-cartantype}. Notably, the classification and the structure of the Nichols algebras of diagonal type are governed by Lie-type objects introduced by Heckenberger \cite{Hweyl}, see also \cite{HY08,AHS,HS}. In particular, they have associated (generalized) root systems and, besides those of Cartan type, we find among the examples root systems of finite-dimensional contragredient Lie superalgebras in characteristic $0$, $3$ and $5$, and root systems of finite-dimensional contragredient Lie algebras in positive characteristic, as it was observed by Andruskiewitsch, Angiono and Yamane \cite{AAYamane,AA-diag-survey}. Consequently, we get small quantum groups with these more general root systems.

\begin{figure}[h]
\begin{tikzpicture}[every text node part/.style={align=center}]

\node[inner sep=10pt, minimum width=5pt] (q) at (-2.5,0) {$\fq\in\mathbb{C}^{\theta\times\theta}$};

\node[inner sep=10pt, minimum width=5pt,label={below: Nichols algebra}] (BV) at (0,0) {$\BV_\fq$};

\node[inner sep=10pt, minimum width=5pt,label={below:Bosonization}] (BVH) at (3.5,0) {$\BV_\fq\#\mathbb{C}\Gamma$};

\node[inner sep=10pt, minimum width=5pt,label={below:Drinfeld double}] (D) at (7,0) {$\cD(\BV_\fq\#\mathbb{C}\Gamma)$};

\node[inner sep=10pt, minimum width=5pt,label={above:Small quantum \\ group}] (uq) at (9.25,0) {$u_\fq$\vphantom{$\cD(\BV_\fq\#\mathbb{C}\Gamma)$}};


\draw [->,decorate,decoration={snake,amplitude=.05cm}] (q) -- (BV);

\draw [->,decorate,decoration={snake,amplitude=.05cm}] (BV) -- (BVH);

\draw [->,decorate,decoration={snake,amplitude=.05cm}] (BVH) -- (D);

\draw[transform canvas={yshift=-1pt},-] (uq) -- (D);
\draw[transform canvas={yshift=1pt},-] (uq) -- (D);

\end{tikzpicture}
\caption{We construct small quantum groups from matrices with finite-dimensional Nichols algebras of diagonal type; $\Gamma$ is a group quotient of $\Z^\theta$. For instance, the positive part of $u_q(\mathfrak{g})$ is the Nichols algebra of $\fq=(q^{d_ic_{ij}})_{i,j}$ where $C=(c_{ij})_{i,j}$ is the Cartan matrix of $\mathfrak{g}$ and $(d_ic_{ij})_{i,j}$ is symmetric. Thus, the corresponding bosonization is the Borel subalgebra and $u_q(\mathfrak{g})$ is a quotient of $u_{\fq}$ by a central group subalgebra.}
\label{fig:uq}
\end{figure}


\subsection{Main results}

Let $u_\fq$ be as in Figure \ref{fig:uq}. Then it admits a triangular decomposition $u_\fq=u_\fq^-u_\fq^0u_\fq^+$ which gives rise to Verma type modules $M(\pi)$ for every algebra map $\pi:u_\fq^0\longrightarrow\mathbb{C}$. Moreover, every $M(\pi)$ has a unique simple quotient $L(\pi)$ and any simple $u_\fq$-module can be obtained in this way. For instance, all these were computed for $\fq$ of type $\mathfrak{ufo}(7)$ in \cite{AAMRufo}. The linkage principle gives us information about the composition factors of the Verma modules as we explain after introducing some notation. 

Let $\{\alpha_1, ..., \alpha_\theta\}$ be the canonical $\Z$-basis of $\Z^\I$ with $\I=\{1, ..., \theta\}$. The matrix $\fq$ defines a bicharacter $\Z^\I\times\Z^\I\longrightarrow\CC^\times$ which we denote also $\fq$. Given $\beta\in\Z^\I$, we set $q_\beta=\fq(\beta,\beta)$ and $b^\fq(\beta)=\ord q_\beta$.  We denote $\rho^\fq:\Z^\I\longrightarrow\mathbb{C}^\times$ the group homomorphism such that $\rho^\fq(\alpha_i)=q_{\alpha_i}$ for all $i\in\I$. We notice $u_\fq^0$ is an abelian group algebra generated by $K_i,L_i$, $i\in\I$ (in Figure \ref{fig:uq}, $\Gamma$ is generated by the $K_i$'s) and unlike $u_q(\mathfrak{g})$, they yield two copies of the (finite) torus. If $\alpha=n_1\alpha_1+\cdots+n_\theta\alpha_\theta\in\Z^\I$, we denote $K_\alpha=K_1^{n_1}\cdots K_\theta^{n_\theta}$ and $L_\alpha=L_1^{n_1}\cdots L_\theta^{n_\theta}$. The algebra map $\pi\widetilde\mu:u_\fq^0\longrightarrow\mathbb{C}$ is defined by $\pi\widetilde\mu(K_\alpha L_\beta)=\frac{\fq(\alpha,\mu)}{\fq(\mu,\beta)}\pi(K_\alpha L_\beta
)$, $\alpha,\beta\in\Z^\I$. Let $\Delta^\fq_+\subset\Z_{\geq0}^\I$ be the set of positive roots of the Nichols algebra $\BV_\fq$. If $\beta\in\Delta^\fq_+$, we define
\begin{align*}
\beta\downarrow\mu=\mu-n^\pi_\beta(\mu)\beta
\end{align*}
where $n^\pi_\beta(\mu)$ is the unique $n\in\{1, ..., b^\fq(\beta)-1\}$ such that 
$q_{\beta}^{n}-\rho^\fq(\beta)\,\pi\widetilde{\mu}(K_{\beta}L_{\beta}^{-1})=0$, if it exists, and otherwise $n^\pi_\beta(\mu)=0$.

\begin{theorem}[Strong linkage principle]\label{teo:main teo}

Let $L$ be a composition factor of $M(\pi\widetilde{\mu})$, $\mu\in\Z^\I$. Then $L\simeq L(\pi\widetilde{\mu})$ or $L\simeq L(\pi\widetilde{\lambda})$ with $\lambda=\beta_r\downarrow\cdots\beta_1\downarrow\mu$ for some $\beta_1, ..., \beta_r\in\Delta_+^\fq$.
\end{theorem}

A first consequence of this principle is that it defines an equivalence relation which completely characterizes the blocks of the category. It is also the starting point to imagine character formulas for the simple modules. In this direction, we  deduce that $M(\pi\widetilde{\mu})=L(\pi\widetilde{\mu})$ is simple if and only if
\begin{align*}
\prod_{\substack{\beta\in\Delta_+^\fq\\ 1\leq t<b^\fq(\beta)}} \left(q_{\beta}^{t}-\rho^\fq(\beta)\,\pi\widetilde{\mu}(K_{\beta}L_{\beta}^{-1})\right)\neq0;
\end{align*}
this was also proved in \cite[Proposition 5.16]{HY}. Moreover, we give a character formula for $L(\pi\widetilde{\mu})$ if the above expression is zero for a unique $\beta\in\Delta_+^\fq$ and for $t=n_\beta^\pi(\mu)$. Explicitly,
\begin{align*}
\ch L(\pi\widetilde\mu)=\quad\frac{1-e^{-n_\beta^\pi(\mu)\beta}}{1-e^{-\beta}} 
\prod_{\gamma\in \Delta_+^\fq\setminus\{\beta\}}\frac{1-e^{-b^\fq(\gamma)\gamma}}{1-e^{-\gamma}}.
\end{align*}
By analogy with the theory of Lie superalgebras \cite{kac,ser}, we say that the number of zeros of the former product measures the degree of atypicality, see \S\ref{subsec:typical}. Like in \cite{Y} where Yamane  gives a Weyl--Kac character formula for typical simple modules over quantum groups of  Nichols algebras of diagonal type with finite root systems (when the Nichols algebra is finite-dimensional, his typical modules coincides with ours).

When I discussed a preliminary version of Theorem \ref{teo:main teo} with Nicol\'as Andruskiewitsch and Simon Riche, they respectively asked me about the Weyl groupoid and the affine Weyl group, how these objects come into play. In fact, the linkage principle is usually encoded in the dot action of the affine Weyl group, and the Weyl groupoid is its replacement  in the theory of Nichols algebras \cite{Hweyl}. We partially answer their questions in the next corollary. We discover also a phenomenon similar to what occurs in the setting of Lie superalgebras. There, the action of the affine Weyl group generated by the even reflections as well as the translations by odd roots take part in the linkage principle, see {\it e.~g.} \cite{chenwang}.

Let $\Delta_{+,\car}^\fq$ be the set of positive Cartan roots of $\fq$ and $s_\beta$ the reflection associated to $\beta\in\Delta_{+,\car}^\fq$. We set $\varrho^\fq=\frac{1}{2}\sum_{\beta\in \Delta^\fq_+}(b^\fq(\beta)-1)\beta$. For $\mu\in\Z^\I$, $\beta\in\Delta_{+,\car}^\fq$ and $m\in\Z$, we define
\begin{align*}
s_{\beta,m}\bullet\mu=s_\beta(\mu+mb^\fq(\beta)\beta-\varrho^\fq)+\varrho^\fq.
\end{align*}
We denote $\cW_{\link}^\fq$ the group generated by all the affine reflections $s_{\beta,m}$. We recall that $\fq$ is of Cartan type if $\Delta_+^\fq=\Delta_{+,\car}^\fq$, and $\fq$ is of super type if its root system is isomorphic to the root system of a finite-dimensional contragredient Lie superalgebra
in characteristic 0. If $\fq$ is of super type, then $\Delta_{+,\odd}^\fq:=\Delta^\fq_+\setminus\Delta_{+,\car}^\fq$ is not empty and $\ord q_\beta=2$ for every root $\beta\in\Delta_{+,\odd}^\fq$; in this case $\beta\downarrow\mu=\mu$ or $\mu-\beta$. 

\begin{corollary}[Linkage principle]\label{cor:main cor}
Assume $\pi$ is the trivial algebra map. Let $L(\pi\widetilde{\lambda})$ be a composition factor of $M(\pi\widetilde{\mu})$. Then
\begin{enumerate}
 \item $\lambda\in\cW_{\link}^\fq\bullet\mu$ if $\fq$ is of Cartan type.
\smallskip
\item $\lambda\in\cW^\fq_{\link}\bullet(\mu+\Z\Delta_{+,\odd}^\fq)$ if $\fq$ is of super type.
\end{enumerate}
\end{corollary}

This corollary holds more generally for matrices of standard type, those with constant bundle of root system. However, there are matrices which are not, as the matrices of modular type. The assumption on $\pi$ is not particularly restrictive as we deal with all the simple modules of highest-weight $\widetilde{\mu}:u_\fq^0\longrightarrow\CC$, $K_\alpha L_\beta\mapsto\frac{\fq(\alpha,\mu)}{\fq(\mu,\beta)}$, for all $\mu\in\Z^\I$. In the case of $u_q(\mathfrak{q})$, these form the category of modules of type $1$ in the sense of Lusztig, cf. \cite[\S2.4]{AJS}.

\subsection{Sketch of the proof}

As we mentioned at the beginning, we imitate the ideas of \cite[\S1-\S7]{AJS}. There, the authors consider $\Z^\I$-graded  algebras admitting a triangular decomposition $U=U^-U^0U^+$ with $U^0$ commutative and satisfying additional conditions which are fulfilled by $u_q(\mathfrak{g})$ and $U^{[p]}(\mathfrak{g}_k)$. Then, given a Noetherian commutative algebra $\bA$ and an algebra map $\pi:U^0\longrightarrow\bA$, they define certain categories $\cC_\bA$ of $\Z^\I$-graded $(U,\bA)$-bimodules. We observe here that we can consider these categories also for $u_\fq$. Roughly speaking, in the case $\bA=\CC$, this identifies with the abelian subcategory generated by the simple modules $L(\pi\widetilde{\mu})$ for $\mu\in\Z^\I$.
 
Powerful tools used in {\it loc.~cit.} to study the categories $\cC_\bA$ are the so-called Lusztig automorphisms $T_w$ of $u_q(\mathfrak{g})$, where $w$ runs in the Weyl group of $\mathfrak{g}$. We find a difference between $u_q(\mathfrak{g})$ and $u_\fq$ at this point. Indeed, we have Lusztig isomorphisms but connecting possibly different algebras, {\it i.e.} $T_w:u_{w^{-*}\fq}\longrightarrow u_{\fq}$ and the matrices $w^{-*}\fq$ and $\fq$ are not necessarily equal. These isomorphisms were defined in \cite{Hlusztigiso} for each $w\in{}^\fq\cW$, the Weyl groupoid of $\fq$ \cite{Hweyl}. Nevertheless, we can carefully use them as in \cite{AJS}.

First, we produce different triangular decompositions on $u_{\fq}$ and then each triangular decomposition gives rise to new Verma modules. Namely, for $w\in{}^\fq\cW$ and $\mu\in\Z^\I$, we denote $Z_\CC^w(\mu)$ the Verma module induced by $\pi\widetilde{\mu}$ using the triangular decomposition $u_\fq=T_w(u_{w^{-*}\fq}^-)\,u_\fq^0\,T_w(u_{w^{-*}\fq}^+)$. Let $L_\CC^w(\mu)$ denote the unique simple quotient of $Z_\CC^w(\mu)$. For instance, $Z_\CC(\mu):=Z_\CC^\id(\mu)=M(\pi\widetilde{\mu})$ and $L_\CC(\mu):=L_\CC^\id(\mu)=L(\pi\widetilde{\mu})$.

Now, we set $\mu\langle w\rangle=\mu+ w(\varrho^{w^{-*}\fq})-\varrho^\fq\in\Z^\I$; we notice that $\varrho^{w^{-*}\fq}$ and $\varrho^\fq$ could be different. We show that the Verma modules $Z^w_\CC(\mu\langle w\rangle)$ and $Z^x_\CC(\mu\langle x\rangle)$ have identical characters and the Hom-space between them is one-dimensional, for all $w,x\in{}^\fq\cW$. Moreover, we construct inductively a generator for each space and compute its kernel. We also prove that the image of $\Phi:Z_\CC(\mu)\longrightarrow Z^{w_0}_\CC(\mu\langle w_{0}\rangle)$ is isomorphic to $L_\CC(\mu)$, cf. Figure \ref{fig:sketch}.

\begin{figure}[h]
\begin{tikzpicture}[scale=.9,every node/.style={scale=0.9}]

\node (0) at (-5.5,0) {$Z_\CC(\mu)$};

\node (s) at (0,0) {$Z^{w}_\CC(\mu\langle w\rangle)$};

\node (s1) at (3.5,0) {$Z^{w\sigma_{i_s}}_\CC(\mu\langle w\sigma_{i_s}\rangle)$};

\node (w0) at (8,0) {$Z^{w_0}_\CC(\mu\langle w_{0}\rangle)$};

\node (sp) at (0,-2) {$Z^{w_s}_\CC(\beta\downarrow\mu\langle w\rangle)$};

\node (0p) at (-5.5,-2) {$Z_\CC(\beta\downarrow\mu)$};

\draw[->,transform canvas={yshift=.3cm}] (0) to [out=5,in=175]  node [above,sloped] {\tiny$\Phi$} node [below,sloped] {\tiny$\circlearrowleft$} (w0);

\draw[->] (0) to (-4,0);

\draw[dotted] (-1.5,0) to (s);

\draw[->] (6,0) to (w0);

\draw[dotted] (5.5,0) to (6,0);

\draw[->] (s) to node [above] {\tiny$\varphi$} (s1);

\draw[->] (sp) to node [left] {\tiny$\psi$} (s);

\draw[->] (0p) to (-3.5,-2);

\draw[dotted] (-2.5,-2) to (sp);

\end{tikzpicture}
\caption{Here $w_0=1^{\fq}\sigma_{i_1}\cdots\sigma_{i_n}$ is a reduced expression of the longest element in ${}^\fq\cW$, $w=1^\fq\sigma_{i_1}\cdots\sigma_{i_{s-1}}$ and $\beta=w\alpha_{i_s}\in\Delta_+^\fq$.  The horizontal maps are generators of the corresponding Hom-spaces and $\Im\psi=\Ker\varphi$.
}
\label{fig:sketch}
\end{figure}

We are ready to outline the last step of the proof of Theorem \ref{teo:main teo}. Let $L_\CC(\lambda)$ be a composition factor of $Z_\CC(\mu)$ not isomorphic to $L_\CC(\mu)$. Then $L_\CC(\lambda)$ is a composition factor of $\Ker\Phi$ and hence also of $\Ker\varphi=\Im\psi$, cf. Figure \ref{fig:sketch} (for some $s$). Since the Verma modules in the same row of Figure \ref{fig:sketch} have identical characters, we conclude that $L_\CC(\lambda)$ is a composition factor of $Z_\CC(\beta\downarrow\mu)$ as well. We observe that $-\beta^\fq_{top}=\sum_{\beta\in \Delta^\fq_+}(b^\fq(\beta)-1)\beta\leq\lambda\leq\beta\downarrow\mu<\mu$ because $\beta^\fq_{top}$ is the maximum $\Z^\I$-degree of $\BV_\fq$. Therefore, by repeating this procedure, we will find $\beta_1, ..., \beta_r\in\Delta_+^\fq$ such that 
$\lambda=\beta_r\downarrow\cdots\beta_1\downarrow\mu$ as desired.

\subsection{Relations with other algebras in the literature}

We would like to remark that the representations of the present small quantum groups can be helpful for other algebras. For instance, this was pointed out by Andruskiewitsch, Angiono and Yakimov \cite[\S1.3.3]{AAYakimov} for the large quantum groups studied in \cite{Ang-reine,Ang-distinguished,AAYakimov} which are analogous to the quantized enveloping algebras of De Concini–Kac–Procesi. The small quantum groups are particular quotients of them. As in \cite{AAYakimov}, we highlight that our results apply to small quantum groups at even roots of unity.

Another example is given by the braided Drinfeld doubles of Nichols algebras recently constructed by Laugwitz and Sanmarco \cite{LS}. These are quotients of $u_\fq$ with only one copy of the finite torus. Since the corresponding projections preserve the triangular decompositions \cite[Proposition 3.16]{LS}, a linkage principle for them can be deduced from Theorem \ref{teo:main teo}.

Finally, it is worth noting that Pan and Shu \cite{panshu} have obtained results comparable to ours, as well as their proofs, for modular Lie superalgebras. This could suggest a relationship between small quantum groups of super type and the corresponding modular Lie superalgebras as it happens between $u_q(\mathfrak{g})$ and $U^{[p]}(\mathfrak{g}_k)$.

\subsection{Organization} The exposition is mostly self-contained. In Section \ref{subsection:conventions} we set up general conventions. In Section \ref{sec:nichols} we collect the main concepts regarding Nichols algebras (PBW basis, Weyl groupoid, root system) and their properties. We illustrate them in super type $A(1|1)$. In Section \ref{sec:Drinfeld} we recall the construction of the Drinfeld doubles and their Lusztig isomorphisms. In Section \ref{sec:AJS} we introduce the categories defined in \cite{AJS} and  summarize their general features. Next we investigate these categories over the Drinfeld doubles of Section \ref{sec:Drinfeld} using their Lusztig isomorphisms: in Section \ref{sec:vermas}, we construct the different Verma modules mentioned previously in this introduction, and we study the morphisms between them in Section \ref{sec:morphisms}. Finally, we prove our main results in Section \ref{sec:linkage}.

\subsection{Acknowledgments} 

I am very grateful to Nicolas Andruskiewitsch, for motivating me to study the representations of Hopf algebras since my PhD, to Ivan Angiono, for patiently answering all my questions about Nichols algebras of diagonal type, and to Simon Riche, for teaching me so much about representation theory. I also thank them for the fruitful discussions. I thank the anonymous referee for their careful reading and valuable comments

\section{Conventions}\label{subsection:conventions}

Throughout our work $\ku$ denotes an algebraically closed field of characteristic zero. For $q\in\ku^\times$ and $n\in\N$, we recall the quantum numbers
\begin{align*}
(n)_q =\sum_{j=0}^{n-1}q^{j} \quad\mbox{and the identity}\quad (n)_{q}=q^{n-1} (n)_{q^{-1}}.
\end{align*}
Let $\theta \in\N$. We set $\I=\I_\theta=\{1, 2,\dots,\theta\}$. We denote $\Pi=\{\alpha_{1}, \dots ,\alpha_{\theta}\}$ the canonical $\Z$-basis of $\Z^\I$. We will write $0=0\alpha_1+\cdots+0\alpha_\theta$. We will use $\Pi$ to identify the matrices of size $\theta\times\theta$ and the bicharacters on $\Z^\I$ with values in $\ku^\times$. Explicitly, given $\fq=(q_{ij})_{i,j\in\I}\in(\ku^\times)^{\I\times\I}$, by abuse of notation, we will denote $\fq:\Z^\I\times\Z^\I\longrightarrow\ku^\times$ the bicharacter defined by
\begin{align*}
\fq(\alpha_i,\alpha_j)=q_{ij}\quad\forall\,i,j\in\I.
\end{align*}
Let $\beta\in\Z^\I$. We will write $q_\beta=\fq(\beta,\beta)$. In particular, $q_{\alpha_i}=q_{ii}$ for all $i\in\I$. The bound function \cite[$(2.12)$]{HY} is
\begin{align}\label{eq: b}
b^{\fq}(\beta)=
\begin{cases}
\min\{m\in\N\mid (m)_{q_\beta}=0\}&\mbox{if $(m)_{q_\beta}=0$ for some $m\in\N$},\\
\infty&\mbox{otherwise}.
\end{cases}
\end{align}
Of course, $b^{\fq}(\beta)$ is finite if and only if $q_\beta$ is a primitive root of $1$ of order $b^{\fq}(\beta)$.

We will consider the dual action of $\Aut_{\Z}(\Z^\I)$ on bicharacters. That is, if $w\in\Aut_\Z(\Z^\I)$, then the bicharacter $w^*{\fq}:\Z^\I\times\Z^\I\longrightarrow\ku^\times$ is defined by
\begin{align*}
w^*{\fq}(\alpha,\beta)={\fq}(w^{-1}(\alpha),w^{-1}(\beta))\quad\forall \alpha,\beta\in\Z^\I.
\end{align*}
The partial order $\leq^w$ in $\Z^\I$ is defined by $\lambda\leq^w\mu$ if and only if $w^{-1}(\mu-\lambda)\in\Z_{\geq0}^\I$. When $w=\id$, we simply write $\leq$.

Let $M=\oplus_{\nu\in\Z^\I} M_{\nu}$ be a $\Z^\I$-graded module over a ring $\bA$. We call {\it weight space} to the homogeneous component $M_\nu$. The support $\sup(M)$ is the set of all $\nu\in\Z^\I$ such that $M_\nu\neq0$. If $w\in\Aut_\Z(\Z^\I)$, the $w${\it-twisted} $M[w]$ is the $\Z^\I$-graded $\bA$-module whose weight spaces are $M[w]_\nu=M_{w^{-1}\nu}$ for all $\nu\in\Z^\I$. This an endofunctor in the category of $\Z^\I$-graded $\bA$-modules. If the weight spaces are $\bA$-free of finite rank, the formal character of $M$ is
\begin{align*}
\ch M=\sum_{\mu\in\Z^\I}\rank_\bA(M_\mu)\,e^\mu
\end{align*}
which belongs to the $\Z$-algebra generated by the symbols $e^\mu$ with multiplication $e^\mu\cdot e^\nu=e^{\mu+\nu}$. We denote $\overline{(-)}$ and $w(-)$ the automorphisms of it given by $\overline{e^\mu}=e^{-\mu}$ and $w(e^\mu)=e^{w\mu}$ for all $\mu\in\Z^\I$ and $w\in\Aut_{\Z}(\Z^\I)$. Therefore
\begin{align}\label{eq:w ch}
\ch(M[w])=w(\ch\, M).
\end{align}

\section{Finite-dimensional Nichols algebras of diagonal type}\label{sec:nichols}

The finite-dimensional Nichols algebras of diagonal type were classified by Heckenberger in \cite{Hroot}\footnote{Indeed, he classifies a more larger family but we only consider here finite-dimensional Nichols algebras.}. They are parameterized by matrices of scalars. The classification and their structure are governed by certain Lie type objects, such as, PBW basis, the generalized Cartan matrix, the Weyl groupoid and the generalized root system, among others. We recall here their main features needed for our work.  We refer to \cite{AA-diag-survey} for an overview of the theory.

We fix $\theta\in\N$ and set $\I=\I_\theta$. Throughout this section $\fq=(q_{ij})_{i,j\in\I}\in(\ku^\times)^{\I\times\I}$ denotes a matrix with finite-dimensional Nichols algebra $\BV_\fq$. This is constructed as a quotient of the free $\ku$-algebra generated by $E_1, ..., E_\theta$, that is
\begin{align*}
\BV_\fq=\ku\langle E_1, ..., E_\theta\rangle/\cJ_\fq
\end{align*}
for certain ideal $\cJ_\fq$. It is a 
$\Z^\I$-graded algebra with 
\begin{align*}
\deg E_i=\alpha_i\quad\forall i\in\I
\end{align*}
and $\Z$-graded with $\deg E_i=1$ for all $i\in\I$. A minimal set of generators for $\cJ_\fq$ was given by Angiono in \cite{Ang-reine, Ang-presentation}. This set consists of powers of the generators of the PBW basis and generalizations of quantum Serre relations. However, we do not need their explicit description here.

\begin{example}\label{ex:small qg}
Let $\mathfrak{g}$ be a finite dimensional semisimple Lie algebra over $\mathbb{C}$ with Cartan matrix $C=(c_{ij})_{i,j\in\I}$ and $D=(d_i)_{i\in\I}$ such that $DC$ is symmetric. Let $q$ be a primitive root of unity and $\fq=(q^{d_ic_{ij}})_{i,j\in\I}$. Then $\BV_\fq$ is finite-dimensional. Moreover, if $\ord q$ is an odd prime, not $3$ if $\mathfrak{g}$ has a component of type $G_2$, then $\BV_\fq$ is the positive part of the small quantum group $u_q(\mathfrak{g})$, one of the algebras analyzed by Andersen--Jantzen--Soergel \cite[Section 1.3. Case 2]{AJS}. 
\end{example}

\begin{example}\label{ex:the example}
Let $q\in\ku$ be a primitive root of order $N>2$ and 
\begin{align*}
\fq=\begin{pmatrix} -1 & -q \\ -1 & -1 \end{pmatrix}.
\end{align*}
Set $E_{12}=\ad_cE_1(E_2)=(E_1E_2-q_{12}E_2E_1)$. Then
\begin{align*}
\BV_\fq=\langle E_1,E_2 \mid E_1^2=E_2^2=E_{12}^N=0\rangle.
\end{align*}
This is a Nichols algebra of super type $A(1|1)$. This term refers to the root system which will be introduced in \S\ref{subsec:root system}; see Example \ref{ex:the example root system}. Through this section we will use this example to illustrate the different concept. See also \cite[\S5.1.11]{AA-diag-survey}.

\end{example}

\subsection{PBW basis}\label{subsec:PBW} In \cite{karchencko}, Karchencko proves the existence of homogeneous elements $E_{\beta_1}, ..., E_{\beta_n}\in\BV_\fq$ with $\deg E_{\beta_\nu}=\beta_\nu\in\Z_{\geq0}^\I$ and $b^\fq(\beta_\nu)<\infty$, recall \eqref{eq: b}, such that
\begin{align*}
\left\{ E_{\beta_{1}}^{m_1}\cdots E_{\beta_{n}}^{m_n}\mid 0\leq m_\nu<b^{\fq}(\beta_\nu),\,1\leq \nu\leq n\right\}
\end{align*}
is a linear basis of $\BV_\fq$. We will see in the next section a way of constructing these elements $E_{\beta_\nu}$. The set of {\it positive roots} of $\fq$ is
\begin{align}\label{eq:positive roots}
\Delta_+^\fq=\{\beta_1, ..., \beta_n\}.
\end{align}
It turns out that the elements of $\Delta_+^\fq$
are pairwise different \cite[Proposition 2.12]{CH}. The set of (positive) {\it simple roots} is $\Pi^\fq=\{\alpha_1, ..., \alpha_\theta\}$. We set
\begin{align}\label{eq:beta top}
\beta_{top}^\fq=\sum_{\beta\in \Delta^\fq_+}(b^\fq(\beta)-1)\beta.
\end{align}
This is the weight of the homogeneous component of maximum $\Z^\I$-degree of $\BV_\fq$.

\begin{example}\label{ex:the example pbw}
The positive roots of $\fq$ in Example \ref{ex:the example} are
\begin{align*}
\Delta_+^\fq=\{\alpha_1,\alpha_1+\alpha_2,\alpha_2\}
\end{align*}
with
\begin{align*}
b^\fq(\alpha_1)=2,\quad
b^\fq(\alpha_1+\alpha_2)=N,\quad
b^\fq(\alpha_2)=2.
\end{align*}
The associated PBW basis is
\begin{align*}
\{E_1^{m_1}E_{12}^{m_2}E_2^{m_3}\mid 0\leq m_1,m_3<2,\,0\leq m_2<N\}.
\end{align*}
\end{example}

\subsection{Weyl groupoid}
For distinct $i,j\in\I$, by \cite{Roso} there exists $m\in\N_0$ such that
\begin{align*}
(m+1)_{q_{ii}}(q_{ii}^mq_{ij}q_{ji}-1)=0.
\end{align*}
Thus, it is possible to define the {\it generalized Cartan matrix} $C^\fq=(c_{ij}^\fq)_{i,j\in\I}$ of $\fq$ as
\begin{align}\label{eq:C chi}
c_{ij}^\fq=\begin{cases}
        2,&\mbox{if $i=j$};\\
        -\min\{m\in\N_0\mid(m+1)_{q_{ii}}(q_{ii}^mq_{ij}q_{ji}-1)=0\},&\mbox{otherwise.}
       \end{cases}
\end{align}
This matrix leads to the reflection $\sigma_i^{\fq}\in\Aut_\Z(\Z^\I)$, $i\in\I$, given by
\begin{align}\label{eq:sigma chi}
\sigma_i^{\fq}(\alpha_j)=\alpha_j-c_{ij}^\fq\alpha_i\quad\forall\,j\in\I.
\end{align}
It turns out that the Nichols algebra of $(\sigma_i^\fq)^*\fq$ is also finite-dimensional for all $i\in\I$. However, $\BV_{(\sigma_i^\fq)^*\fq}$ is not necessarily isomorphic to $\BV_\fq$. 

Let $\cH$ be the family of matrices with finite-dimensional Nichols algebras and $r_i:\cH\longrightarrow\cH$ the bijection given by
\begin{align}\label{eq:r i}
r_i(\fq)=(\sigma_i^{{\fq}})^*{\fq}.
\end{align}
for all $\fq\in\cH$ and $i\in\I$.
For instance,
$
r_p(\fq)(\alpha_i,\alpha_j)=q_{ij}q_{ip}^{-c_{pj}^{\fq}}q_{pj}^{-c_{pi}^{\fq}}q_{pp}^{c_{pi}^{\fq}c_{pj}^{\fq}}
$ for all $p,i,j\in\I$ and hence
\begin{align}\label{eq: C invariante}
c_{ij}^{r_i(\fq)}=c_{ij}^\fq.
\end{align}
It is immediate that $\sigma_i^{r_i(\fq)}=\sigma_i^{\fq}$ and $r_i^2(\fq)=\fq$. We notice that
\begin{align*}
r_{i_k}(\cdots (r_{i_1}(\fq)))=\left(\sigma_{i_k}^{r_{i_{k-1}}\cdots r_{i_{1}}(\fq)}\cdots \sigma_{i_2}^{r_1(\fq)}\sigma_{i_1}^\fq\right)^*\fq.
\end{align*}
We denote $\cG=\langle r_i\mid i\in\I\rangle$ which is a subgroup of the group of bijections on $\cH$.

Let $\cX$ be a $\cG$-orbit. The {\it Weyl groupoid} $\cW$ of $\cX$ is the category whose objects are the matrices belong to $\cX$, and their morphisms are generated by the arrows 
$\sigma_i^{{\fq}}:{\fq}\rightarrow r_i({\fq})$ for all $\fq\in\cX$ and $i\in\I$. We denote $1^{\fq}$ the identity of $\fq$ and set $\cW_\fq=\Hom_{\cW}(\fq,\fq)$. Thus, a morphism in $\cW$ is of the form
\begin{align*}
\sigma_{i_k}^{r_{i_{k-1}}\cdots r_{i_{1}}(\fq)}\cdots \sigma_{i_2}^{r_1(\fq)}\sigma_{i_1}^\fq&:\fq\longrightarrow r_{i_k}(\cdots (r_{i_1}(\fq))).&
\end{align*}
To shorten notation, we observe that a morphism in $\cW$ is determined either by specifying the source,
$\sigma_{i_k}\cdots\sigma_{i_1}^{\fq}:\fq\longrightarrow r_{i_k}\cdots r_{i_1}(\fq)$, or by specifying the target, $1^{\fq}\sigma_{i_k}\cdots\sigma_{i_1}:r_{i_1}\cdots r_{i_k}(\fq)\longrightarrow\fq$. For $w=\sigma_{i_k}\cdots\sigma_{i_1}^{\fq}$, we will write 
\begin{align*}
w^{*}\fq=r_{i_k}(\cdots (r_{i_1}(\fq)))\quad\mbox{and}\quad w^{-*}\fq=(w^{-1})^*\fq.
\end{align*}
We set ${}^\fq\cW$, resp. $\cW^\fq$, the family of morphisms in $\cW$ whose target, resp. source, is $\fq$.

Clearly, $(\sigma_p^{\fq})^*\fq(\sigma_p^{\fq}(\alpha),\sigma_p^{\fq}(\alpha))=\fq(\alpha,\alpha)$. Hence, for all $w\in\cW^\fq$ and $\alpha\in\Z^\I$, it holds
\begin{align}\label{eq: b invariante}
b^{w^*\fq}(w\alpha)=b^{\fq}(\alpha). 
\end{align}

By \cite[Theorem 1]{HY08}, the defining relations of the Weyl groupoid are of Coxeter type:
\begin{align*}
(\sigma_i\sigma_i)1^{\fq}=1^{\fq}\quad\mbox{and}\quad(\sigma_i\sigma_j)^{m_{ij}^{\fq}}1^{\fq}=1^{\fq}\quad\forall \fq\in\cX,\,\forall i,j\in\I,\,i\neq j;
\end{align*}
for certain exponents $m_{ij}^{\fq}$. Similar to Coxeter groups, there exists a length function $\ell:\cW\longrightarrow\N_0$ given by
\begin{align*}
\ell(w)=\min\bigl\{k\in\N_0\mid \exists i_1, \dots, i_k\in\I,\fq\in\cX:w=\sigma_{i_k}\cdots\sigma_{i_1}^{\fq}\bigr\} 
\end{align*}
for $w\in\cW$.  Also, for each $\fq\in\cX$, there exists a unique morphism $w_0\in{}^\fq\cW$ of maximal length. Let $w_0=1^\fq\sigma_{i_1}\cdots\sigma_{i_n}$ be a reduced expression of the longest element in ${}^\fq\cW$. The set of positive roots can be constructed as follows
\begin{align}\label{eq:roots are conjugate to simple}
\Delta_+^\fq=\biggl\{\beta_\nu=1^\fq\sigma_{i_1}\cdots\sigma_{i_{\nu-1}}(\alpha_{i_\nu})\mid1\leq\nu\leq n\biggr\},
\end{align}
see for instance \cite[Proposition 2.12]{CH}.

\begin{example}\label{ex:the example groupoid}
Let $\fq$ be as in Example \ref{ex:the example}. We set 
\begin{align*}
\fp=
\left(
\begin{array}{cc}
    -1 & q^{-1} \\ 1 & q
       \end{array}
       \right)
       \quad\mbox{and}\quad
       \fr=
\left(
\begin{array}{cc}
    q & q^{-1} \\ 1 & -1
       \end{array}
       \right)
       . 
\end{align*}
Then $\cX=\{\fq,\fp,\fr,\fq^t,\fp^t,\fr^t\}$ is a $\cG$-orbit. Indeed, the generalized Cartan matrices of them are all equal to
\begin{align*}
\left(
\begin{array}{rr}
    2 & -1 \\ -1 & 2
       \end{array}
       \right)
\end{align*}
and the Weyl groupoid of $\cX$ is
\begin{center}
\begin{tikzpicture}

\node (q) at (0,0) {$\fq$};
\node (p) at (-5,0) {$\fp$};
\node (r) at (5,0) {$\fr$};

\node (qt) at (0,-1.5) {$\fq^t$};
\node (pt) at (-5,-1.5) {$\fp^t$};
\node (rt) at (5,-1.5) {$\fr^t$};

\draw[->,transform canvas={yshift=1pt}] (q) to [out=5,in=175]  node [above,sloped] {\tiny$\sigma_2^\fq$} (r);

\draw[->,transform canvas={yshift=-1pt}] (r) to [out=-175,in=-5]  node [below,sloped] {\tiny$\sigma_2^\fq$} (q);

\draw[<-,transform canvas={yshift=1pt}] (p) to [out=5,in=175]  node [above,sloped] {\tiny$\sigma_1^\fq$} (q);

\draw[->,transform canvas={yshift=-1pt}] (p) to [out=-5,in=-175]  node [below,sloped] {\tiny$\sigma_1^\fp$} (q);

\draw[->,transform canvas={yshift=1pt}] (qt) to [out=5,in=175]  node [above,sloped] {\tiny$\sigma_2^{\fq^t}$} (rt);

\draw[->,transform canvas={yshift=-1pt}] (rt) to [out=-175,in=-5]  node [below,sloped] {\tiny$\sigma_2^{\fq^t}$} (qt);

\draw[<-,transform canvas={yshift=1pt}] (pt) to [out=5,in=175]  node [above,sloped] {\tiny$\sigma_1^{\fq^t}$} (qt);

\draw[->,transform canvas={yshift=-1pt}] (pt) to [out=-5,in=-175]  node [below,sloped] {\tiny$\sigma_1^{\fp^t}$} (qt);

\draw[->,transform canvas={xshift=-2pt}] (p) to [out=-95,in=95]  node [left] {\tiny$\sigma_2^{\fp}$} (pt);

\draw[->,transform canvas={xshift=2pt}] (pt) to [out=85,in=-85]  node [right] {\tiny$\sigma_2^{\fp^t}$} (p);

\draw[->,transform canvas={xshift=-2pt}] (r) to [out=-95,in=95]  node [left] {\tiny$\sigma_1^{\fp}$} (rt);

\draw[->,transform canvas={xshift=2pt}] (rt) to [out=85,in=-85]  node [right] {\tiny$\sigma_1^{\fp^t}$} (r);


%
%
\end{tikzpicture}
\end{center}
The corresponding Nichols algebras are not necessarily isomorphic. For instance,
\begin{align*}
\BV_\fp&=\langle E_1,E_2\mid E_1^2=E_2^N=E_{221}^2=0\rangle\quad\mbox{with}\quad E_{221}=(\ad_cE_{2})^2(E_1),\\
\BV_\fr&=\langle E_1,E_2\mid E_1^N=E_2^2=E_{112}^2=0\rangle\quad\mbox{with}\quad E_{112}=(\ad_cE_{1})^2(E_2).
\end{align*}
\end{example}

\subsection{Root systems}\label{subsec:root system}
The bundle $\cR=\{\Delta^\fq\}_{\fq\in\cX}$ with
\begin{align*}
\Delta^{\fq}=\Delta^{\fq}_+\cup-\Delta^{\fq}_+ 
\end{align*}
is the so-called {\it (generalized) root system of $\cX$} (or of $\fq\in\cX$)

We highlight that, unlike classical root systems, $\Delta^\fq$ and $\Delta^{\fp}$ are not necessarily equals for distinct $\fq,\fp\in\cX$. When it is needed, we will write $\beta^\fq$ to emphasize that $\beta\in\Z^\I$ belongs to $\Delta^\fq$. However, they share other characteristics with classical root systems which will be useful, cf. \cite{CH,HY}. In particular, for all $i\in\I$, it holds that
\begin{align}\label{eq:sigma Delta = Delta}
\sigma_i^{\fq}(\Delta_+^{\fq}\setminus\{\alpha_i\})=\Delta_+^{r_i(\fq)}\setminus\{\alpha_i\},\quad\sigma_i^{\fq}(\alpha_i)=-\alpha_i
\quad\mbox{and}\quad w(\Delta^{w^{-*}\fq})=\Delta^{\fq}.
\end{align}
Also, $\ell(w)=\left|w(\Delta_+^{\fp})\cap-\Delta_+^{\fq}\right|$ for any $w:\fp\rightarrow\fq$ \cite[Lemma 8 $(iii)$]{HY08}. This implies that
\begin{align}
w_0(\beta_{top}^{w_0^{-*}\fq})=-\beta_{top}^\fq. 
\end{align}

\begin{example}\label{ex:the example root system}
The generalized root system of $\fq$ from Example
\ref{ex:the example} is constant: $\Delta^\fz=\{\pm\alpha_1,\pm(\alpha_1+\alpha_2),\pm\alpha_2\}$ for all $\fz\in\cX$. 

We have introduced here generalized root systems for Nichols algebras. However, this is a combinatorial object appearing in other context. According to \cite[Theorem 2.34]{AA-diag-survey} contragredient Lie superalgebras have associated  generalized root systems. For instance, that associated to $\mathfrak{sl}(1|1)$ is equal to $\cR$.
\end{example}

\subsection{A shift operation on weights}
Recall the element $\beta_{top}^\fq$ in \eqref{eq:beta top}. We set
\begin{align}\label{eq:varrho}
\varrho^\fq=\frac{1}{2}\beta^\fq_{top}.
\end{align}
This element will play the role of the semi-sum of the positive roots in Lie theory. The following definition is analogous to \cite[4.7$(4)$]{AJS}.

\begin{definition}\label{def:mu w}
For $\mu\in\Z^\I$ and $w:w^{-*}\fq\rightarrow\fq$ in $\cW$, we set
\begin{align*}
\mu\langle w\rangle=\mu+ w(\varrho^{w^{-*}\fq})-\varrho^\fq\in\Z^\I.
\end{align*}
\end{definition}

We observe that 
\begin{align}\label{eq:w varrho - varrho}
w(\varrho^{w^{-*}\fq})-\varrho^\fq=-\sum_{\beta\in \Delta_+^\fq\,:\,w^{-1}\beta\in \Delta_-^{w^{-*}\fq }}(b^\fq(\beta)-1)\beta
\end{align}
since $w(\Delta^{w^{-*}\fq})=\Delta^{\fq}$ and $b^{w^{-*}\fq}(w^{-1}\beta)=b^{\fq}(\beta)$. For instance, $\mu\langle w_0\rangle=\mu-\beta^\fq_{top}$. This operation satisfies that
\begin{align}\label{eq:mu w s}
\mu\langle w\sigma_i\rangle=\mu\langle w\rangle-(b^\fq(w\alpha_{i})-1)w\alpha_{i}
\end{align}
where $w\sigma_i:\sigma_i^*w^{-*}\fq\rightarrow\fq$ and $i\in\I$. Indeed,
\begin{align*}
w\sigma_i\left(\varrho^{\sigma_i^*w^{-*}\fq}\right)-\varrho^\fq&=w\left(\sigma_i(\varrho^{\sigma_i^*w^{-*}\fq})-\varrho^{w^{-*}\fq}+\varrho^{w^{-*}\fq}\right)-\varrho^\fq\\
&=w\left(-(b^{w^{-*}\fq}(\alpha_i)-1)\alpha_i+\varrho^{w^{-*}\fq}\right)-\varrho^\fq\\
&=w\left(\varrho^{w^{-*}\fq}\right)-\varrho^\fq-\left(b^{w^{-*}\fq}(\alpha_i)-1\right)w\alpha_i.
\end{align*}

\begin{example}\label{ex:the example varrho}

Even if the generalized root system is constant, the corresponding elements $\varrho^\fq$ could be different. For instance, 
\begin{align*}
\varrho^\fq=\frac{1}{2}\left(\alpha_1+(N-1)(\alpha_1+\alpha_2)+\alpha_2\right)\\
\varrho^\fp=\frac{1}{2}\left(\alpha_1+(\alpha_1+\alpha_2)+(N-1)\alpha_2\right)\\
\varrho^\fr=\frac{1}{2}\left((N-1)\alpha_1+(\alpha_1+\alpha_2)+\alpha_2\right)
\end{align*}
where $\fq$, $\fp$ and $\fr$ are as in Example \ref{ex:the example groupoid}.

\end{example}

\section{Small quantum groups}\label{sec:Drinfeld}

In this section, we introduce the Hopf algebras which we are interested in. We will follow \cite{Hlusztigiso}. We fix $\theta\in\N$ and a matrix $\fq=(q_{ij})_{i,j\in\I}\in(\ku^\times)^{\I\times\I}$ with finite-dimensional Nichols algebra $\BV_\fq$, where $\I=\I_\theta$. Let $\cX$ be the $\cG$-orbit of $\fq$ and $\cW$ its Weyl groupoid.

We denote $U_\fq$ the Hopf algebra generated by the symbols $K_i$, $K_i^{-1}$, $L_i$, $L_i^{-1}$, $E_i$ and $F_i$, with $i\in\I$, subject to the relations:
\begin{align*}
K_iE_j=q_{ij}E_jK_i,\quad&
L_iE_j=q_{ji}^{-1}E_jL_i\\
K_iF_j=q_{ij}^{-1}F_jK_i,\quad&
L_iF_j=q_{ji}F_jL_i
\end{align*}
\begin{align}
\label{eq: EF FE}
E_iF_j-F_jE_i&=\delta_{i,j}(K_i-L_i)\\
\notag
XY&=YX,\\
\notag
K_iK_i^{-1}&=L_iL_i^{-1}=1
\end{align}
for all $i,j\in\I$ and $X,Y\in\{K_i^{\pm1},L_i^{\pm1}\mid i\in\I\}$. Also, the generators $E_i$ are subject to the defining relations $\cJ_\fq$ of the Nichols algebra, and the generators $F_i$ are subject to the relations $\tau(\cJ_\fq)$, cf. \eqref{eq:tau} below. It turns out that $\cJ_\fq$ coincides with the radical of a natural Hopf pairing between the free algebras generated by the elements $E_i$ and $F_i$. However, these latter relations, as well as the counit, comultiplication, and antipode of $U_\fq$,  will not be needed.

We have that $U_\fq=\oplus_{\mu\in\Z^\I} \,(U_\fq)_\mu$ is a $\Z^\I$-graded Hopf algebra with
\begin{align*}
\deg E_i=-\deg F_i=\alpha_i\quad\mbox{and}\quad \deg K_i^{\pm1}=\deg L_i^{\pm1}=0\quad\forall\,i\in\I.
\end{align*}
For $\alpha=n_1\alpha_1+\cdots+n_\theta\alpha_\theta\in\Z^\I$, we set 
\begin{align*}
K_\alpha=K_1^{n_1}\cdots K_\theta^{n_\theta}\quad\mbox{and}\quad L_\alpha=L_1^{n_1}\cdots L_\theta^{n_\theta}.
\end{align*}
In particular, $K_{\alpha_i}=K_i$ for $i\in\I$.

Given $\mu\in\Z^\I$, we define $\fq_\mu\in\Alg(U_\fq^0,\ku)$ and $\widetilde{\mu}\in\Aut_{\ku-alg}(U_\fq^0)$ by
\begin{align}
\label{eq:tilde mu}
\fq_\mu(K_\alpha L_\beta)=\frac{\fq(\alpha,\mu)}{\fq(\mu,\beta)}
\quad\mbox{and}\quad
\widetilde{\mu}(K_\alpha L_\beta)=\fq_\mu(K_\alpha L_\beta)\,K_\alpha L_\beta
\end{align}
for all $\alpha,\beta\in\Z^\I$. Then, it follows easily that 
\begin{align}
s\,u=u\,\widetilde\mu(s)\quad\forall u\in (U_\fq)_\mu,\,s\in U_\fq^0
\end{align}
and
\begin{align}\label{eq:composition tilde}
\widetilde{\nu}\circ\widetilde{\mu}=\widetilde{\nu+\mu}\quad\forall\nu,\mu\in\Z^\I.
\end{align}

The multiplication of $U_\fq$ induces a linear isomorphism
\begin{align}
U_\fq^-\ot U_\fq^0\ot U_\fq^+\longrightarrow U_\fq
\end{align}
where
\begin{align*}
U_\fq^{+}=\ku\langle E_i\mid i\in\I\rangle\simeq\BV_\fq,\quad
U_\fq^{0}=\ku\langle K_i,L_i\mid i\in\I\rangle,
\quad
U_\fq^{-}=\ku\langle F_i\mid i\in\I\rangle
\end{align*}
are subalgebras of $U_\fq$. We notice that   $U_\fq^0$ identifies with the group algebra $\ku(\Z^\I\times\Z^\I)$ for any matrix $\fq$.

\begin{definition}\label{def:small quantum group}
Let $p:U_\fq\longrightarrow u_\fq$ be a $\Z^\I$-graded algebra projection and set $u_\fq^{\pm,0}:=p(U_\fq^{\pm,0})$. We call $u_\fq$ {\it small quantum group} if the multiplication $u_\fq^-\ot u_\fq^0\ot u_\fq^+\longrightarrow u_\fq$ induces a linear isomorphism and $u_\fq^\pm\simeq U_\fq^\pm$.
\end{definition}

We say they are ``small'' because $u_\fq^\pm$ is finite-dimensional but we do not make any assumption on the size of $u_\fq^0$. Thus, $U_\fq$ is a small quantum group.

\begin{example}
Let $\fq$ be as in Example \ref{ex:small qg}. Then $U_\fq/\langle K_i-L_i^{-1}, K_i^{2\ord q}-1\mid i\in\I\rangle$ is a small quantum group. Moreover, if $\ord q$ is an odd prime, not $3$ if $\mathfrak{g}$ has a component of type $G_2$, then $u_q(\mathfrak{g})\simeq U_\fq/\langle K_i-L_i^{-1}, K_i^{2\ord q}-1\mid i\in\I\rangle$.
\end{example}

\begin{example}\label{ex:uq}
Let us explain how to construct the small quantum group $u_\fq$ of Figure \ref{fig:uq}.
 By \cite[Corollary 5.9]{Hlusztigiso}, there is a skew Hopf pairing $\eta$ between $U^+_\fq\#\ku\langle K_i\mid i\in\I\rangle$ and $U^-_\fq\#\ku\langle L_i\mid i\in\I\rangle$, and the corresponding Drinfeld double is isomorphic to $U_\fq$. Let $\Gamma=\overline{\langle K_i\mid i\in\I\rangle}$ be a group quotient and set $g_i=\overline{K_i}$, $i\in\I$. Suppose the character $\chi_i:\Gamma\longrightarrow\ku^\times$, $\chi_i(g_j)=q_{ij}$ for all $i,j\in\I$, is well-defined and set $\widetilde{\Gamma}=\langle \chi_i\mid i\in\I\rangle$. Then $\langle L_i\mid i\in\I\rangle\longrightarrow\widetilde{\Gamma}$, $L_i\mapsto\chi_i$, is a group quotient. Moreover, $\eta$ induces a pairing between $U^+_\fq\#\ku\Gamma$ and $U^-_\fq\#\ku\widetilde{\Gamma}$, and the corresponding Drinfeld double is $u_\fq$.
\end{example}

\begin{example}
The braided Drinfeld doubles introduced in \cite[\S3.2]{LS} are small quantum groups.
\end{example}

\subsection{Lusztig isomorphisms}\label{subsec:lusztig iso}

In \cite[Theorem 6.11]{Hlusztigiso}, Heckenberger constructs certain algebra isomorphisms
\begin{align}\label{eq:Ti}
T_i=T_i^{(\sigma_i^\fq)^*\fq}:U_{(\sigma_i^\fq)^*\fq}\longrightarrow U_\fq
\end{align}
for all $i\in\I$. They emulate some properties of the Lusztig automorphisms of small quantum groups.  We emphasize that these isomorphisms depend on the matrix defining the Drinfeld double of the domain, but we will omit this in the notation when no confusion can arise. We do not need the precise definition of these functions. We just recall some useful properties for us.

Let $w:w^{-*}\fq\rightarrow\fq$ be a morphism in $\cW$. We choose a reduced expression $w=1^\fq\sigma_{i_k}\cdots\sigma_{i_1}$ and denote
\begin{align}
T_w=T_{i_k}\cdots T_{i_1}:U_{w^{-*}\fq}\longrightarrow U_\fq. 
\end{align}
This isomorphism depends on our chosen  expression for $w$. However, if $w=1^\fq\sigma_{j_k}\cdots\sigma_{j_1}$ is another reduced expression, there exists $\underline{a}=(a)_{i\in\I}\in(\ku^{\times})^\I$ such that
\begin{align}\label{eq:Tw different presentation}
T_{i_k}\cdots T_{i_1}=T_{j_k}\cdots T_{j_1}
\varphi_{\underline{a}}
\end{align}
where $\varphi_{\underline{a}}$ is the algebra automorphism of $U_{w^{-*}\fq}$ given by 
\begin{align*}
\varphi_{\underline{a}}(K_i)=K_i,\quad
\varphi_{\underline{a}}(L_i)=L_i,\quad
\varphi_{\underline{a}}(E_i)=a_iE_i\quad\mbox{and}\quad
\varphi_{\underline{a}}(F_i)=a_i^{-1}F_i\quad\forall i\in\I.
\end{align*}
Indeed, both reduced expression of $w$ can be transformed each other using only the Coxeter type relations by \cite[Theorem 5]{HY08}. Then, \eqref{eq:Tw different presentation} follows using \cite[Theorem 6.19 and Proposition 6.8 $(ii)$]{Hlusztigiso}.

These isomorphisms permute the weight spaces according in the following way:
\begin{align}\label{eq:Tw U alpha es U w alpha}
T_w\left((U_{w^{-*}\fq})_{\alpha}\right)=(U_\fq)_{w\alpha}
\end{align}
for all $\alpha\in\Z^\I$, cf. \cite[Proposition 4.2]{HY}. They also have a good behavior on the middle subalgebras $U^0_{w^{-*}\fq}= U^0_\fq$. Explicitly,
\begin{align}\label{eq:Tw restricted to U0}
T_w(K_\alpha L_\beta)=K_{w\alpha}L_{w\beta}
\end{align}
for all $\alpha,\beta\in\Z^\I$ because $T_i^{\pm1}(K_j)=K_{\sigma_i^\fq(\alpha_j)}$ and $T_i^{\pm1}(L_j)=L_{\sigma_i^\fq(\alpha_j)}$ by definition \cite[Lemma 6.6]{Hlusztigiso}.
It follows that 
\begin{align}\label{eq: mu conjugado por Tw}
T_w\circ\widetilde\mu\circ T_w^{-1}{}_{|U^0_\fq}=\widetilde{w(\mu)}
\end{align}
for all $\mu\in\Z^\I$; keep in mind that $\widetilde{\mu}$ depends on $w^{-*}\fq$ and $\widetilde{w(\mu)}$ depends on $\fq$.

For the longest element $w_0$, there is a permutation $f$ of $\I$ such that $w_0^{-*}\fq(\alpha_i,\alpha_j)=f^{*}\fq(\alpha_i,\alpha_j)=\fq(\alpha_{f(i)},\alpha_{f(j)})$. Also, there is $\underline{b}=(b)_{i\in\I}\in(\ku^{\times})^\I$ such that
\begin{align}\label{eq:T w0}
T_{w_0}=\phi_1\,\varphi_f\,\varphi_{\underline{b}} 
\end{align}
by \cite[Corollary 6.21]{Hlusztigiso}; where $\phi_1$ is the algebra automorphism of $U_{\fq}$ given by
\begin{align*}
\phi_{1}(K_i)=K_i^{-1},
\quad
\phi_{1}(L_i)=L_i^{-1},
\quad
\phi_{1}(E_i)=F_iL_i^{-1}\quad\mbox{and}\quad
\varphi_{1}(F_i)=K_i^{-1}E_i
\end{align*}
for all $i\in\I$ and $\varphi_f:U_\fq\longrightarrow U_{f^{*}\fq}$ is the algebra isomorphism given by
\begin{align*}
\varphi_f(K_i)=K_{f(i)},\quad 
\varphi_f(L_i)=L_{f(i)},\quad
\varphi_f(E_i)=E_{f(i)},\quad\mbox{and}\quad
\varphi_f(F_i)=F_{f(i)}
\end{align*}
for all $i\in\I$.

Let $\tau$ be the algebra antiautomorphism of $U_\fq$ defined by
\begin{align}\label{eq:tau}
\tau(K_i)=K_i,\quad
\tau(L_i)=L_i,\quad
\tau(E_i)=F_i\quad\mbox{and}\quad
\tau(F_i)=E_i
\end{align}
for all $i\in\I$, see \cite[Proposition 4.9 $(7)$]{Hlusztigiso}. Notice that $\tau^2=\id$.

The generators of the PBW basis of the Nichols algebras can be constructed using the Lusztig isomorphisms as follows. Let $w_0=1^\fq\sigma_{i_1}\cdots\sigma_{i_n}$ be a reduced expression of the longest element in ${}^\fq\cW$ and recall from \eqref{eq:roots are conjugate to simple} that
$\beta_\nu=1^\fq\sigma_{i_1}\cdots\sigma_{i_{\nu-1}}(\alpha_{i_\nu})$, $1\leq\nu\leq n$, are the positive roots of $\fq$. We set
\begin{align}
E_{\beta_\nu}=T_{i_1}\cdots T_{i_{\nu-1}}(E_{i_\nu})\in (U^+_\fq)_{\beta_\nu}
\quad\mbox{and}\quad
F_{\beta_\nu}=T_{i_1}\cdots T_{i_{\nu-1}}(F_{i_\nu})\in (U^-_\fq)_{-\beta_\nu},
\end{align}
for all $1\leq\nu\leq n$; by an abuse of notation $E_{i_\nu}$ and $F_{i_\nu}$ denote the generators of $U_{(\sigma_{i_{\nu-1}} \cdots \sigma_{i_1})^*\fq}$. These elements depend on the reduced expression of $w_0$. 

By \cite[Theorem 4.9]{HY}, we know that 
\begin{align}
\label{eq:PBW basis}
\begin{split}
\left\{ E_{\beta_{f(1)}}^{m_1}\cdots E_{\beta_{f(n)}}^{m_n}\right.&\left.\mid 0\leq m_\nu<b^\fq(\beta_\nu),\,1\leq\nu\leq n\right\}\quad\mbox{and}\\
&\left\{ F_{\beta_{f(1)}}^{m_1}\cdots F_{\beta_{f(n)}}^{m_n}\mid 0\leq m_\nu<b^\fq(\beta_\nu),\,1\leq\nu\leq n\right\}
\end{split}
\end{align}
are linear bases of $U_\fq^+$ and $U_\fq^-$ for any permutation $f$ of $\I$.
It is immediate that
\begin{align}\label{eq:ch U-}
\ch\, U^-_\fq=
\prod_{\beta\in \Delta_+^\fq}\frac{1-e^{-b^\fq(\beta)\beta}}{1-e^{-\beta}}
=\prod_{\beta\in \Delta_+^\fq}\left(1+e^{-\beta}+\cdots+e^{(1-b^\fq(\beta))\beta}\right).
\end{align}
We point out that the  weight space of minimum degree of $U^-_\fq$ is one-dimensional and spanned by
\begin{align}
F^\fq_{top} =F_{\beta_1}^{b^\fq(\beta_1)-1}\cdots F_{\beta_n}^{b^\fq(\beta_n)-1}.
\end{align}

\begin{example}\label{ex:the example U}
Keeping the notation of Example \ref{ex:the example groupoid}, $T_1^\fp:U_\fp\rightarrow U_\fq$ is defined by
\begin{align*}
T_1^\fp(E_1)&=F_1L_1^{-1},\quad T_1^\fp(E_2)=E_{12},\quad T_1^\fp(K_\alpha L_\beta)=K_{\sigma_1^\fp(\alpha)}L_{\sigma_1^\fp(\alpha)},\\ T_1^\fp(F_1)&=K_1^{-1}E_1,\quad T_1^\fp(F_2)=\frac{1}{q-1}(F_1F_2+F_2F_1).
\end{align*}
\end{example}


\subsection{Parabolic subalgebras} 

We fix $i\in\I$ and denote $U_\fq(\alpha_i)$ and $U_\fq(-\alpha_i)$ the subalgebras of $U_\fq$ generated by $E_i$ and $F_i$, respectively. We set 
\begin{align}\label{eq:P alpha i}
P_\fq(\alpha_i)=U_\fq(-\alpha_i)U^0_\fq U^+_\fq.
\end{align}
This is a subalgebra of $U_\fq$ thanks to the defining relations.

By the definition of the Lusztig isomorphisms, the restriction $T_i:P_{\sigma_i^{*}\fq}(\alpha_i)\longrightarrow P_\fq(\alpha_i)$ is an isomorphism and we can decompose $P_\fq(\alpha_i)$ as
\begin{align}\label{eq:P alpha i otra}
P_{\fq}(\alpha_i)=U_\fq(\alpha_i)\,U^0_\fq \,T_i(U^+_{(\sigma_i^\fq)^{*}\fq}).
\end{align}
Indeed, $T_i(E_i)=F_iL_i^{-1}$, $T_i(E_j)=\ad_c^{-c_{ij}^\fq} E_i(E_j)$ and $T_i(F_i)=K_i^{-1}E_i$ \cite[Lemma 6.6]{Hlusztigiso}.

\subsection{Some definitions} We introduce some elements which will be key in the analysis of the representations of $U_\fq$. 

\begin{definition}
Let $\beta\in\Delta^\fq$ and $n\in\N$. We set
\begin{align*}
[\beta;n]=(n)_{q_{\beta}^{-1}}K_\beta-(n)_{q_{\beta}}L_\beta
\quad\mbox{and}\quad
[n;\beta]=(n)_{q_{\beta}^{-1}}L_\beta-(n)_{q_{\beta}}K_\beta.
\end{align*}
\end{definition}

It follows from the defining relations that
\begin{align}\label{eq: E Fn}
E_iF^n_i=F_i^nE_i+F_i^{n-1}[\alpha_i;n]
\quad\mbox{and}\quad
F_iE^n_i=E_i^nF_i+E_i^{n-1}[n;\alpha_i].
\end{align} 
for all $i\in\I$. Moreover, once we have fixed a PBW basis as in \eqref{eq:PBW basis}, we can apply the corresponding Lusztig isomorphisms to the above identities and obtain that
\begin{align*}
E_\beta F^n_\beta=F_\beta^nE_\beta+F_\beta^{n-1}[\beta;n]\quad\mbox{and}\quad 
F_\beta E^n_\beta=E_\beta^nF_\beta+E_\beta^{n-1}[n;\beta]
\end{align*}
for all $\beta\in\Delta^\fq_+$.

\begin{definition}\label{def:t beta}
Given $\beta\in\Delta^\fq$, $\mu\in\Z^\I$ and a $U_\fq^0$-algebra $\bA$ with structural map $\pi:U_\fq^0\longrightarrow\bA$, we define $t_\beta^\pi(\mu)$ as the unique $t\in\{1, ..., b^\fq(\beta)-1\}$ such that 
\begin{align*}
1=q_{\beta}^{1-t}\pi\tilde\mu(K_\beta L_\beta^{-1}),
\end{align*}
if it exists, and otherwise $t^\pi_\beta(\mu)=0$.
\end{definition}

Equivalently, we see can say that, modulo $b^\fq(\beta)$, $t_\beta^\pi(\mu)$ is the unique $t\in\{1, ..., b^\fq(\beta)\}$  such that $\pi\tilde\mu([\beta;t])=0$. In fact, $\pi\tilde\mu([\beta;b^\fq(\beta)])=0$ and
\begin{align}\label{eq:[beta;n] = otra}
[\beta;t]=(t)_{q_{\beta}}L_\beta\left(q_{\beta}^{1-t}K_\beta L_\beta^{-1}-1\right).
\end{align}

We also observe that
\begin{align}\label{eq:[n;beta] = otra}
[t;\beta]=(t)_{q_{\beta}^{-1}}L_\beta
\left(1-q_{\beta}^{t-1}K_\beta L_\beta^{-1}\right).
\end{align}

Given a $U_\fq^0$-algebra $\bA$ with structural map $\pi:U_\fq^0\longrightarrow\bA$ and a morphism $w\in\cW^\fq$, we denote by $\bA[w]$ the $U_\fq^0$-algebra with structural map $\pi[w]:U_\fq^0\longrightarrow\bA$ defined by
\begin{align}\label{eq:pi[w]}
\pi[w](K_\alpha L_\alpha)=\pi(K_{w^{-1}\alpha} L_{w^{-1}\beta})
\end{align}
for all $\alpha,\beta\in\Z^\I$. We highlight that $\pi[w]=\pi\circ T_w^{-1}{}_{|U_\fq^0}$ for any Lusztig isomorphism $T_w:U_\fq\longrightarrow U_{w^{*}\fq}$ associated to $w$ by \eqref{eq:Tw restricted to U0}.

If $\beta=w\alpha_i\in\Delta^{w^*\fq}$ for some $i\in\I$, it holds that
\begin{align}\label{eq: t w alpha i}
t_{w\alpha_i}^{\pi[w]}(w\mu)=t_{\alpha_i}^\pi(\mu).
\end{align}
In fact, we have that
\begin{align}
\notag
\pi[w]\widetilde{w\mu}([\beta,t])
&=\pi[w]\left((t)_{w^{*}\fq(\beta,\beta)^{-1}}w^{*}\fq(\beta,w\mu)K_{\beta}-(t)_{w^{*}\fq(\beta,\beta)}w^{*}\fq(w\mu,-\beta)L_{\beta}\right)\\
\label{eq:pi w mu beta}
&=\pi\left((t)_{q_{\alpha_i}^{-1}}\fq(\alpha_i,\mu)K_{\alpha_i}-(t)_{q_{\alpha_i}}\fq(\mu,-\alpha_i)L_{\alpha_i}\right)\\
\notag
&=\pi\tilde\mu[\alpha_i,t].
\end{align}

As in \cite[Definition 2.16]{HY}, we define the group homomorphism $\rho^\fq:\Z^\I\longrightarrow\ku^\times$ by $\rho^\fq(\alpha_i)=\fq(\alpha_i,\alpha_i)$ for all $i\in\I$.

\begin{definition}\label{def:n beta}
Given $\beta\in\Delta^\fq$, $\mu\in\Z^\I$ and a $U_\fq^0$-algebra $\bA$ with structural map $\pi:U_\fq^0\longrightarrow\bA$, we define $n_\beta^\pi(\mu)$ as the unique $n\in\{1, ..., b^\fq(\beta)-1\}$ such that 
\begin{align*}
q_{\beta}^{n}=\rho^\fq(\beta)\,\pi\widetilde{\mu}(K_{\beta}L_{\beta}^{-1}),
\end{align*}
if it exists, and otherwise $n^\pi_\beta(\mu)=0$.
\end{definition}

The above numbers are related in the following way.

\begin{lemma}\label{le:n beta = t beta}
If $\beta=w\alpha_i\in\Delta^{\fq}$ for some $i\in\I$, then
$n_\beta^\pi(\mu)=t_\beta^\pi(\mu\langle w\rangle)$.
\end{lemma}

\begin{proof}
 
We first claim that
\begin{align}\label{eq:rho = ...}
\rho^\fq(w\alpha_i)=\fq(w\alpha_i,w\alpha_i)\,\fq(0\langle w\rangle,w\alpha_i)\,
\fq(w\alpha_i,0\langle w\rangle).
\end{align}
We prove it by induction on the length of $w$. If $w=1^\fq$, then $n_{\alpha_i}^\pi(\mu)=t_{\alpha_i}^\pi(\mu)$ by \eqref{eq:[beta;n] = otra}. We now assume the equality holds for all bicharacters and morphisms in $\cW$ of length $r$. Thus, if $w=\sigma_jw_1$ with $j\in\I$ and $\ell(\sigma_jw_1)=1+\ell(w_1)=1+r$, similar to \eqref{eq:w varrho - varrho}, one can check that
\begin{align*}
0\langle\sigma_jw_1\rangle=-\sum_{\gamma\in \Delta_+^{\sigma_j^{-*}(\fq)}:w_1^{-1}\gamma\in \Delta_-^{(\sigma_jw_1)^{-*}\fq}}(b^\fq(\sigma_j\gamma)-1)\sigma_j\gamma+(b^\fq(\alpha_j)-1)\sigma_j\alpha_j.
\end{align*}
Therefore
\begin{align*}
\fq(\sigma_jw_1\alpha_i,\sigma_jw_1\alpha_i)&\,\fq(0\langle\sigma_j w_1\rangle,\sigma_jw_1\alpha_i)\,
\fq(\sigma_jw_1\alpha_i,0\langle \sigma_jw_1\rangle)=\\
&=\sigma_j^{-*}(\fq)(w_1\alpha_i,w_1\alpha_i)\,
\sigma_j^{-*}(\fq)(0\langle w_1\rangle,w_1\alpha_i)\,\sigma_j^{-*}(\fq)(w_1\alpha_i,0\langle w_1\rangle)\times\\
&\quad\quad\quad \sigma_j^{-*}(\fq)(\alpha_j,w_1\alpha_i)^{b^{\sigma_j^{-*}(\fq)}(\alpha_j)-1}\,\sigma_j^{-*}(\fq)(w_1\alpha_i,\alpha_j)^{b^{\sigma_j^{-*}(\fq)}(\alpha_j)-1}\\
&\overset{(\star)}{=}\rho^{\sigma_j^{-*}(\fq)}(w_1\alpha_i)\,\frac{\rho^\fq(\sigma_jw_1\alpha_i)}{\rho^{\sigma_j^{-*}(\fq)}(w_1\alpha_i)}=\rho^\fq(\sigma_jw_1\alpha_i);
\end{align*}
$(\star)$ follows from the inductive hypothesis and \cite[Lemma 2.17]{HY}. This concludes the induction and our claim holds.

In consequence, we have that
\begin{align*}
q_{\beta}\pi\widetilde{\mu\langle w\rangle}(K_{\beta}L_{\beta}^{-1})&=\fq(\beta,\beta)\fq(0\langle w\rangle,\beta)\fq(\beta,0\langle w\rangle)\,\pi\widetilde{\mu}(K_{\beta}L_{\beta}^{-1})\\
&=\rho^\fq(\beta)\,\pi\widetilde{\mu}(K_{\beta}L_{\beta}^{-1})
\end{align*}
which implies the lemma.
\end{proof}

\section{Andersen--Jantzen--Soergel categories}\label{sec:AJS}

In \cite[Section 2]{AJS}, Andersen, Jantzen and Soergel have defined certain categories of modules for algebras fulfilling most of the more remarkable features of the small quantum groups at roots if unity. We will call them {\it AJS categories}. They consider any $\Z^\I$-graded $\ku$-algebra $U$ endowed with a triangular decomposition
\begin{align*}
U^-\ot U^0\ot U^+\longrightarrow U,
\end{align*}
{\it i.e.} this is a $\ku$-linear isomorphism induced by the multiplication and $U^-$, $U^0$ and $U^+$ are $\Z^\I$-graded subalgebras satisfying the following properties, cf. \cite[\S1.1 and \S2.1]{AJS}:
\begin{align}
\label{eq:property 0}
&U^0\subset U_0,\quad (U^\pm)_0=\ku;\\
\label{eq:property order}
&(U^\pm)_{\nu}\neq0\Rightarrow \pm\nu\geq0;\\
\label{property A} 
&\mbox{$\sup(U^\pm)$ is finite};\\
\label{property B C} 
&\mbox{Each $(U^\pm)_\nu$, and hence $U^\pm$, is finite-dimensional over $\ku$}.
\end{align}
They also assume that $U^0$ is commutative and the existence of a group homomorphism $\Z^\I\longrightarrow\Aut_{\ku-alg}(U^0)$, $\mu\mapsto\widetilde{\mu}$, such that
\begin{align}\label{property D} 
s\,u=u\,\widetilde\mu(s)\quad\forall u\in U_\mu,\,s\in U^0. 
\end{align}
From \eqref{eq:property 0}, we deduce that there are augmentation maps $U^\pm\longrightarrow (U^\pm)_0=\ku$. We denote both of them by $\varepsilon$.

\begin{example}\label{ex:Uq satisface AJS}
A small quantum group satisfies all the previous properties.  
\end{example}

We fix a Noetherian commutative $U^0$-algebra $\bA$ with structural map $\pi:U^0\longrightarrow\bA$. We now present the AJS category $\cC_\bA$, cf. \cite[2.3]{AJS}. An object of $\cC_\bA$ is a $\Z^\I$-graded $U\ot\bA$-module $M$, or equivalently a left $U$-module and right $\bA$-module, such that
\begin{align}
\label{eq: cCA fin gen}
&\mbox{$M$ is finitely generated over $\bA$;}
\\
\label{eq: cCA is A graded}
&\mbox{$M_\mu\bA\subset M_\mu$,}
\\
\label{eq: cCA is U graded}
&\mbox{$U_\nu M_\mu\subset M_{\nu+\mu}$,}
\\
\label{eq:compatibilidad sm ma}
&\mbox{$s\,m=m\,\pi\widetilde{\mu}(s)$,}
\end{align}
for all $\mu,\nu\in\Z^\I$, $m\in M_\mu$ and $s\in U^0$. The last compatibility means that the $U^0$-action is determined by the $\bA$-action and the $\Z^\I$-grading. The morphisms between two objects are the morphisms of $\Z^\I$-graded $U\ot\bA$-modules.

The authors also defined categories $\cC'_\bA$ and $\cC''_\bA$ in similar fashion by replacing $U$ with $U^0U^+$ and $U^0$, respectively. There are obvious induced functors $\cC''_\bA\longrightarrow\cC'_\bA$ and $\cC'_\bA\longrightarrow\cC_\bA$ which are left adjoints of the forgetful functors. Moreover, the categories $\cC_\bA$ and $\cC'_\bA$ have enough projectives \cite[Lemma 2.7]{AJS}.

We next summarize the most important attributes of the AJS categories.

\subsection{Verma modules}
Let $\mu\in\Z^\I$. We denote $\bA^\mu$ the free right $\bA$-module of rank one concentrated in degree $\mu$ and generated by the symbol $\m{\mu}=\m{\mu}_\bA$.

We consider $\bA^\mu$ as an object of $\cC'_\bA$ with left $U^+$-action given by the augmentation map and left $U^0$-action determined by \eqref{eq:compatibilidad sm ma}. Explicitly,
\begin{align}\label{eq:Amu}
su\cdot\m{\mu}=\varepsilon(u)\,\pi\widetilde{\mu}(s)\,\m{\mu}\quad\forall s\in U^0,\,u\in U^+.
\end{align}
We call the induced modules
\begin{align}\label{eq:Z}
Z_\bA(\mu)=U\ot_{U^0U^+}\bA^\mu
\end{align}
{\it Verma modules}; where $U$ and $\bA$ act by left and right multiplication on the left and right factor, respectively. It is isomorphic to $U^-\ot \bA$ as $U^-\ot\bA$-module and its weight spaces are
\begin{align}
Z_\bA(\mu)_\beta=(U^-)_{\beta-\mu}\ot\m{\mu}
\end{align}
for all $\beta\in\Z^\I$. Therefore
\begin{align}\label{eq:ch Z}
\ch\,Z_\bA(\mu)=e^\mu\,\ch\,U^-.
\end{align}
By an abuse of notation, we also denote $\m{\mu}$ the generator $1\ot\m{\mu}$ of $Z_\bk(\mu)$.

A {\it $Z$-filtration} of a module in $\cC_\bA$ is a filtration whose subquotient are isomorphic to Verma modules. In \cite[\S2.11-\S2.16]{AJS} is proven that projective modules admit $Z$-filtrations. Several other properties of these modules are proven in \cite[Sections 2-4]{AJS}.

\subsection{Simple modules}\label{subsec:simples}

Assume that $\bA=\bk$ is a field. Then $Z_\bk(\mu)$ has a unique simple quotient denoted $L_\bk(\mu)$ \cite[\S4.1]{AJS}. This object is characterized as the unique (up to isomorphism) simple {\it highest-weight module} $L$ in $\cC_\bk$, that is $L$ is generated by some $v\in L_\mu$ with $(U^+)_\nu v=0$ for all $\nu>0$. We say that $v$ is {\it highest-weight vector of weight $\mu$}. Moreover, each simple module in $\cC_\bk$ is isomorphic to a unique simple highest-weight module. This characterization of the simple modules implies that their characters are linearly independent. 

We notice that all modules have composition series of finite length. For $M\in\cC_\bk$, $[M:L_\bk(\lambda)]$ denotes the number of composition factors isomorphic to $L_\bk(\lambda)$. Two important properties of the Verma modules are 
\begin{align*}
[Z_\bk(\mu):L_\bk(\mu)]=1\quad\mbox{and}\quad[Z_\bk(\mu):L_\bk(\lambda)]\neq0\Rightarrow\lambda\leq\mu.
\end{align*}

The next lemma is standard. It says that we can read the composition factors of a module from its character. In particular, modules with equal characters have the same composition factors. 

\begin{lemma}\label{lem:ch and composition}
Let $M\in\cC_\bk$. It holds that $\ch\, M=\sum_\lambda a_\lambda\,\ch\, L_\bk(\lambda)$ if and only if $a_\lambda=[M:L_\bk(\lambda)]$ for all $\lambda\in\Z^\I$. \qed
\end{lemma}

\begin{example}\label{example:L one dimensional}
Let $U=U_\fq$ be a Drinfeld double, $\bk=\ku$ and $\mu=0$. Then $K_\alpha L_\beta\cdot\m{0}=\pi(K_\alpha L_\beta)\m{0}$ and hence $Z_\ku(0)=\cM^\fq(\pi)$ is the Verma module of \cite[Definition 5.1]{HY} and $L_\ku(0)=L^\fq(\pi)$ its quotient as in \cite[$(5.7)$]{HY}.
\end{example}

\begin{example}
$L_\bk(\mu)=\bk^\mu$ is one-dimensional if and only if
\begin{align}
\pi\tilde{\mu}(K_iL_i^{-1})=1 \quad\forall i\in\I.
\end{align}
Indeed, using \eqref{eq: EF FE} and \eqref{eq:Amu}, $\bk^\mu$ is the  simple quotient of $Z_\bk(\mu)$ if and only if 
\begin{align*}
0=(K_i-L_i)\cdot\m{\mu}=\left(\fq(\alpha_i,\mu)\pi(K_i)-\frac{1}{\fq(\mu,\alpha_i)}\pi(L_i)\right)\m{\mu}
\end{align*}
which is equivalent to our claim.
\end{example}

\subsection{Blocks}

Let $\sim_b$ denote the smallest equivalence relation in $\Z^\I$ such that $\lambda\sim_b\mu$ if $\Hom_{\cC_\bA}(Z_\bA(\lambda),Z_\bA(\mu))\neq0$ or $\Ext^1_{\cC_\bA}(Z_\bA(\lambda),Z_\bA(\mu))\neq0$. The equivalence classes of $\sim_b$ are called blocks \cite[\S6.9]{AJS}. In case $\bA=\bk$ is a field, this definition coincides with the usual definition of blocks via simple modules \cite[Lemma 6.12]{AJS}. Namely, $\lambda$ and $\mu$ belong to the same block if $L_\bk(\lambda)$ and $L_\bk(\mu)$ have a non trivial extension.

Let $\cD_\bA$ denote the full subcategory of $\cC_\bA$ of all objects admitting a $Z$-filtration. If $b$ is a block, $\cD_\bA(b)$ is the full subcategory of all object admitting a $Z$-filtration whose factors are $Z_\bA(\mu)$ with $\mu\in b$. Likewise $\cC_\bA(b)$ denotes the full subcategory of all objects in $\cC_\bA$ which are the homomorphic image of an object in $\cD_\bA(b)$. The abelian categories $\cD_\bA$ and $\cC_\bA$ decompose into the sum $\oplus_b\cD_\bA(b)$ and $\oplus_b\cC_\bA(b)$. These and other properties are proved in \cite[6.10]{AJS}. We will also call the subcategories $\cD_\bA(b)$ and $\cC_\bA(b)$ blocks. The {\it principal block}, denoted $\cC_{\bk}(0)$, is the block containing $L_\bk(0)$.

\subsection{Duals}

Let $M\in\cC_\bA$. Then $M^\tau=\Hom_\bA(M,\bA)$ is an object in $\cC_\bA$ with $U$-action
\begin{align*}
(uf)(m)=f(\tau(u)m),
\end{align*}
for all $m\in M$ and $f\in M^\tau$, and homogeneous components
\begin{align*}
(M^\tau)_\lambda=\left\{f\in\Hom_\bA(M,\bA)\mid f(M_\mu)=0\,\forall\mu\neq\lambda\right\}
\end{align*}
for all $\lambda\in\Z^\I$.
If $M$ is $\bA$-free, then
\begin{align}\label{eq:ch duals}
\ch(M^\tau)=\ch(M).
\end{align}
From the characterization of the simple objects in $\cC_\bk$, $\bk$ is a field, we deduce that
\begin{align*}
L_\bk(\mu)^\tau\simeq L_\bk(\mu).
\end{align*}
for all $\mu\in\Z^\I$.

\subsection{AJS categories versus usual module categories}\label{subsec:finite} 
Assume that $U^0$ is finite-dimensional, and hence so is $U$. By forgetting the right $\bA$-action, we obtain a fully faithful exact functor from the AJS categories $\cC_\bA$ to the category ${}_{U}\cG$ of finite-dimensional $\Z^\I$-graded $U$-modules. Roughly speaking, the next proposition says that we know all the simple objects and the blocks of the category ${}_U\cG$ if we know the simple module $L_{\ku}(0)$ and the principal block $\cC_{\ku}(0)$ for all the algebra maps $\pi:U^0\longrightarrow\ku$.

\begin{proposition}
Every block in ${}_{U}\cG$ is equivalent as an abelian category to the principal block of some AJS category $\cC_\ku$.
\end{proposition}

\begin{proof}
Since $U$ is finite-dimensional, one can construct Verma and simple modules in ${}_{U}\cG$ like in the AJS categories, see for instance \cite{BT}. Given an algebra map $\pi:U^0\longrightarrow\ku$, let us denote by $\Delta(\pi)$ and $L(\pi)$ the corresponding Verma module and simple modules in ${}_{U}\cG$. Namely, $L(\pi)$ coincides with the image under the forgetful functor of the simple object $L_{\ku}(0)$ of the AJS category $\cC_\ku$ associated to $\pi$. We claim that every simple module in ${}_{U}\cG$ belonging to the same block of $L(\pi)$ is isomorphic to the image of $L_{\ku}(\mu)$ under the forgetful functor for some $\mu$ in the principal block of $\cC_{\ku}$. In fact, suppose that $L(\mu)$ is a simple module in ${}_{U}\cG$ and
\begin{align*}
0\longrightarrow L(\mu)\longrightarrow M\longrightarrow L(\pi)\longrightarrow0 
\end{align*}
is a non trivial extension. Let $m_\mu$ be a highest-weight vector in $L(\mu)$ and $m\in M$ such that its image in $L(\pi)$ is a highest-weight vector. Then $m$ generates $M$ and there is $u_\mu\in U_\mu$ such that $m_\mu=u_\mu \cdot m$. Let $s\in U^0$. Hence 
\begin{align*}
s\cdot m_\mu=su_\mu\cdot m\overset{\eqref{property D}}{=}u_\mu\widetilde{\mu}(s)\cdot m
\overset{\eqref{eq:compatibilidad sm ma}}{=}(\pi\widetilde{\mu})(s)\,m_\mu.
\end{align*}
This implies that $L(\mu)$ is isomorphic to the image of $L_\ku(\mu)$ under the forgetful functor. Therefore every simple module of ${}_U\cG$ belonging to the same block of $L(\pi)$ also is in the image of the principal block, and the proposition follows.
\end{proof}

\subsection{AJS categories and quotients}\label{subsec:quotients}

Let $p:U\longrightarrow\overline{U}$ be a $\Z^\I$-graded algebra projection and set   $\overline{U}^{\pm,0}:=p(U^{\pm,0})$. Assume that $\overline{U}^-\ot \overline{U}^0\ot \overline{U}^+\longrightarrow \overline{U}$ induces a linear isomorphism, $\overline{U}^\pm\simeq U^\pm$ and $\widetilde{\mu}$ induces an algebra automorphism on $\overline{U}^0$. Thus, given an algebra map $\overline{\pi}:\overline{U}^0\longrightarrow\bk$ we can consider the corresponding AJS category which we denote $\overline{\cC}_\bk$; here $\bk$ is a field. We write $\overline{Z}_\bk(\mu)$ and $\overline{L}_\bk(\mu)$ for the Verma and simple modules in $\overline{\cC}_\bk$. If $\pi=\overline{\pi}\circ p:U^0\longrightarrow\bk$, we get an obvious functor $\Inf_{\overline{U}}^U:\overline{\cC}_\bk\longrightarrow\cC_\bk$ such that $\Inf_{\overline{U}}^U(\overline{Z}_\bk(\mu))\simeq Z_\bk(\mu)$ and $\Inf_{\overline{U}}^U(\overline{L}_\bk(\mu))\simeq L_\bk(\mu)$ for all $\mu\in\Z^\I$, by the assumptions on $p$. Then $\ch\overline{Z}_\bk(\mu)=\ch Z_\bk(\mu)$ and $\ch\overline{L}_\bk(\mu)=\ch L_\bk(\mu)$ and therefore
\begin{align}\label{eq:quotients}
\left[\overline{Z}_\bk(\mu):\overline{L}_\bk(\lambda)\right]=[Z_\bk(\mu):L_\bk(\lambda)]
\end{align}
for all $\mu,\lambda\in\Z^\I$ by Lemma \ref{lem:ch and composition}.  Moreover, if $\widetilde{\mu}$ descends to an algebra automorphism of $\overline{U}^0$, we have an analogous result to \eqref{eq:quotients} by considering the category ${}_{\overline{U}}\cG$ instead of $\overline{\cC}_\ku$.

\section{Twisted Verma modules}\label{sec:vermas}

Through this section we fix a matrix $\fq=(q_{ij})_{i,j\in\I}\in(\ku^\times)^{\I\times\I}$ with finite-dimensional Nichols algebra $\BV_\fq$, where $\I=\I_\theta$. Let $\cX$ be the $\cG$-orbit of $\fq$ and $\cW$ its Weyl groupoid. Let $U_\fq$ be the Hopf algebra introduced in Section \ref{sec:Drinfeld}. In the sequel $\bA$ denotes a Noetherian commutative ring and $\bk$ denotes a field, with no connection between them. We assume that both are algebras over $U^0$ and, by abuse of notation, we denote their structural maps $\pi:U_\fq^0\longrightarrow\bA$ and $\pi:U_\fq^0\longrightarrow\bk$. We denote
$\cC_\bA^\fq$, $\cC'^\fq_\bA$ and $\cC''^\fq_\bA$ the corresponding AJS categories; and similarly over the field $\bk$. We point out that the results regarding to simple modules of the following subsections only hold for the AJS categories over the field $\bk$, cf. \S\ref{subsec:simples}, and the results being valid over $\bA$ obviously hold over $\bk$.

In this section we construct and study different Verma modules for $U_\fq$ using the Lusztig isomorphisms of \S\ref{subsec:lusztig iso}. To do this, we mimic the ideas from \cite{AJS} where the authors use the Lusztig automorphisms \cite{L-qgps-at-roots}. Although our demonstrations are almost identical to \cite{AJS}, it is worthwhile to do them thoroughly again as we have morphisms between possibly different algebras parameterized by $\cX$ and permuted by the action of the Weyl groupoid.

\subsection{\texorpdfstring{$w$}{w}-Verma modules} Let $w:w^{-*}\fq\rightarrow\fq$ be a morphism in $\cW$ with $w=1^\fq\sigma_{i_k}\cdots\sigma_{i_1}$ a reduced expression. We consider the Lusztig isomorphism
\begin{align}
T_w=T_{i_k}\cdots T_{i_1}:U_{w^{-*}\fq}\longrightarrow U_\fq. 
\end{align}
Thus, the triangular decomposition of $U_{w^{-*}\fq}$ induces a new triangular decomposition on $U_\fq$. Explicitly,
\begin{align}\label{eq:triangular decomposition w recall}
T_w(U^-_{w^{-*}\fq})\ot U^0_\fq\ot T_w(U^+_{w^{-*}\fq})\longrightarrow U_\fq
\end{align}
since $T_w(U^0_{w^{-*}\fq})=U^0_{\fq}$. 

Given $\mu\in\Z^\I$, we conisder $\bA^\mu$ as a $U^0_\fq\,T_w(U^+_{w^{-*}\fq})$-module with action
\begin{align}\label{eq:Amu twisted}
su\cdot\m{\mu}=\varepsilon(u)\,\pi(\widetilde{\mu}(s))\,\m{\mu}\quad\forall s\in U_\fq^0,\,u\in T_w(U^+_{w^{-*}\chi})
\end{align}
and imitating \cite[\S4.3]{AJS}, we introduce the module
\begin{align}\label{eq:Verma twisted}
Z^w_\bA(\mu)=U_\fq\ot_{U_\fq^0T_w(U^+_{w^{-*}\fq})}\bA^\mu
\end{align}
which belongs to $\cC^\fq_\bA$. We call it {\it $w$-Verma module}. Of course, $Z_\bA^{1^\fq}(\mu)=Z_\bA(\mu)$. We notice that $Z^w_\bA(\mu)$ does not depend on the reduced expression of $w$. In fact, if $\tilde T_w$ is the Lusztig isomorphism associated to another reduced expression of $w$, then $T_w(U^\pm_{w^{-*}\fq})=\tilde T_w(U^\pm_{w^{-*}\fq})$ by \eqref{eq:Tw different presentation}. Hence the decomposition \eqref{eq:triangular decomposition w recall} does not depend on the reduced expression. Moreover, \eqref{eq:Amu twisted} defines the same module for both expressions, and {\it a posteriori} isomorphic Verma modules.

\begin{lemma}
Let $\mu\in\Z^\I$ and $w\in{}^\fq\cW$. Then
\begin{align*}
\ch Z_\bA^w(\mu)=e^{\mu-(w(\varrho^{w^{-*}\fq})-\varrho^\fq)}\ch U^-_\fq.
\end{align*} 
\end{lemma}

\begin{proof}
By the definition, we see that $\ch\,Z_\bA^w(\mu)=e^\mu\,\ch\,T_w(U^-_{w^{-*}\fq})$. On the other hand,  $T_w(U^-_{w^{-*}\fq})\simeq(U_{w^{-*}\fq}^-)[w]$ as $\Z^\I$-graded object by \eqref{eq:Tw U alpha es U w alpha}. Let us write $\Delta_+^\fq=R_1\cup R_2$ with
\begin{align*}
R_1=\left\{\beta\in \Delta_+^\fq\,:\,w^{-1}\beta\in\Delta_-^{w^{-*}\fq}\right\}
\quad\mbox{and}\quad
R_2=\left\{\beta\in\Delta_+^\fq\,:\,w^{-1}\beta\in\Delta_+^{w^{-*}\fq}\right\}.
\end{align*}
Therefore 
\begin{align}\label{eq:chTwU-}
\begin{split}
\ch\,T_w(U^-_{w^{-*}\fq})\overset{\eqref{eq:w ch}}{=}w(\ch\,U^-_{w^{-*}\fq})&\overset{\eqref{eq:ch U-}}{=}
\prod_{\gamma\in \Delta_+^{w^{-*}\fq}}\frac{1-e^{-b^{w^{-*}\fq}(\gamma)w\gamma}}{1-e^{-w\gamma}}\\
&=
\prod_{\beta\in R_1}\frac{1-e^{b^{\fq}(\beta)\beta}}{1-e^{\beta}}
\cdot
\prod_{\beta\in R_2}\frac{1-e^{-b^{\fq}(\beta)\beta}}{1-e^{-\beta}}\\
&=
\prod_{\beta\in R_1}e^{(b^{\fq}(\beta)-1)\beta}
\cdot
\prod_{\beta\in\Delta_+^\fq}\frac{1-e^{-b^{\fq}(\beta)\beta}}{1-e^{-\beta}}\\
&\overset{\eqref{eq:w varrho - varrho}}{=}e^{\varrho^\fq-w(\varrho^{w^{-*}\fq})}\ch U^-_\fq
\end{split}
\end{align}
and the lemma follows.
\end{proof}

As a consequence, we have that
\begin{align}\label{eq: ch Zmu = ch Zw muw}
\ch Z_\bA(\mu)=\ch Z_\bA^w(\mu\langle w\rangle).
\end{align}
for all $\mu\in\Z^\I$ and $w\in{}^\fq\cW$ with $\mu\langle w\rangle=\mu-(w(\varrho^{w^{-*}\fq})-\varrho^\fq)$, recall Definition \ref{def:mu w}. Moreover, we next see that the Hom spaces between these $w$-Verma modules are free of rank one over $\bA$, cf. \cite[Lemma 4.7]{AJS}.

\begin{proposition}\label{prop:Hom Zx y Zw}
Let $\mu\in\Z^\I$ and $x,w\in{}^\fq\cW$. Then
\begin{align*}
\Hom_{\cC_\bA^\fq}(Z_\bA^x(\mu\langle x\rangle),Z_\bA^w(\mu\langle w\rangle))\simeq \bA\quad\mbox{and}\quad
\Hom_{\cC_\bA^\fq}(Z_\bA^x(\mu\langle x\rangle),Z_\bA^w(\mu\langle w\rangle)^\tau)\simeq \bA.
\end{align*} 
\end{proposition}

\begin{proof}
Let us abbreviate $M^x=Z_\bA^x(\mu\langle x\rangle)$. We have that $(M^x)_{\mu\langle x\rangle}$ is $\bA$-free of rank $1$ and $\mu\langle x\rangle+x\beta$ is not a weight for any $\beta\in \Delta_+^{x^{-*}\fq}$ because $\ch\,M^x=e^{\mu\langle x\rangle}\ch\,T_x(U^-_{x^{-*}\fq})$. By \eqref{eq: ch Zmu = ch Zw muw}, these claims are also true for $M^w$. Let $v$ be a generator  of $(M^w)_{\mu\langle x\rangle}$ as $\bA$-module. Then  $T_x(U^+_{x^{-*}\fq})v=0$. Therefore, there is a morphism $M^x\rightarrow M^w$, induced by $\m{\mu\langle x\rangle}\mapsto v$ and any other morphism has to be a multiple of it. This shows the first isomorphism. The second one follows similarly by using \eqref{eq:ch duals}.
\end{proof}

\begin{corollary}\label{le:twisted verma in the same block}
Let $\mu\in\Z^\I$ and $w\in{}^{\fq}\cW$. Then $Z^w_\bA(\mu\langle w\rangle)$ and $Z_\bA(\mu)$ belong to the same block. 
\end{corollary}

\begin{proof}
It follows as \cite[Lemma 6.11]{AJS} using the previous proposition.
\end{proof}

\subsection{Twisted simple modules}
Let $w\in\cW$. The new triangular decomposition \eqref{eq:triangular decomposition w recall} on $U_\fq$ satisfies \eqref{eq:property 0}--\eqref{property D} with respect to the $w$-twisted $\Z^\I$-grading on $U_\fq$ and the partial order $\leq^w$; \eqref{property D} holds thanks to \eqref{eq: mu conjugado por Tw}. Similar to \cite[\S4.3]{AJS}, this observation ensures that the $w$-Verma modules satisfy most of the properties of the usual Verma modules. For instance, we can construct the $w$-Verma modules via induction functors,  they have a simple head, and the projectives modules admit $Z^w$-filtrations.

Therefore, in the case of the field $\bk$, the $w$-Verma module $Z^w_\bk(\mu)$ has a unique simple quotient which we denote $L^w_\bk(\mu)$ for all $\mu\in\Z^\I$. As the simple modules in $\cC_\bk^\fq$ are determined by their highest-weights, $w$ induces a bijection in $\Z^\I$, $\mu\leftrightarrow\mu_w$, such that
\begin{align}\label{eq:Lw L mu w}
L^w_\bk(\mu)\simeq L_\bk(\mu_w).
\end{align}
Notice that if $w_0\in{}^\fq\cW$ is the longest element, then $\mu_{w_0}$ is the {\it lowest-weight} of $L_\bk(\mu)$ by \eqref{eq:T w0}; {\it i.e.} $(U^-)_\nu\cdot L_\bk(\mu)_{\mu_{w_0}}=0$ for all $\nu<0$.

We next give more information about the $w$-Verma modules over a the field $\bk$.

\begin{lemma}[{\cite[Lemma 4.8]{AJS}}]\label{le:socle of Zw}
Let $\mu\in\Z^\I$ and $w\in{}^\fq\cW$.
Then the socle of $Z_\bk^w(\mu)$ is a  simple module in $\cC_\bk^\fq$. Furthermore, the element $T_w(F_{top}^{w^{-*}\fq})\m{\mu}$ generates the socle and spans the homogeneous component of weight $\mu-w(\beta_{top}^{w^{-*}\fq})$. 
\end{lemma}

\begin{proof}
The socle of $T_w(U^-_{w^{-*}\fq})$ as a module over itself is spanned by $T_w(F_{top}^{w^{-*}\fq})$ whose degree is $-w(\beta_{top}^{w^{-*}\fq})$. Then any simple submodule of $Z_\bk^w(\mu)$ in $\cC^\fq_\bk$ should contains $T_w(F_{top}^{w^{-*}\fq})\m{\mu}$ and the lemma follows.
\end{proof}

Morphisms as in the lemma below exist by Proposition \ref{prop:Hom Zx y Zw}. Recall that $w_0$ denotes the longest element in ${}^\fq\cW$.

\begin{lemma}[{\cite[Lemma 4.9]{AJS}}]\label{le:imagen de Z w0 mu w0 en Z mu}
Let $\mu\in\Z^\I$,
\begin{align*}
\Phi: Z_\bk(\mu)\longrightarrow Z_\bk^{w_0}(\mu\langle w_0\rangle) \quad\mbox{and}\quad\Phi':Z_\bk^{w_0}(\mu\langle w_0\rangle)\longrightarrow Z_\bk(\mu)
\end{align*}
be non-zero morphisms in $\cC^\fq_\bk$. Then
\begin{align*} 
L_\bk(\mu)\simeq
\soc Z_\bk^{w_0}(\mu\langle w_0\rangle)=\Im\Phi
\quad\mbox{and}\quad
L_\bk^{w_0}(\mu\langle w_0\rangle)\simeq\soc Z_\bk(\mu)=\Im\Phi'.
\end{align*}
\end{lemma}

\begin{proof}
As $\Phi'$ is graded, $\Phi'$ sends $\m{\mu\langle w_0\rangle}$ to the generator of the socle of $Z_\bk(\mu)$ by Lemma \ref{le:socle of Zw}. This shows the second isomorphism and the second one follows in a similar way.
\end{proof}

\begin{example}\label{example:L one dimensional twisted}
$L^{w_0}_\bk(\mu\langle w_0\rangle)=\bk^{\mu\langle w_0\rangle}$ is one-dimensional if and only if
\begin{align}
\pi\widetilde{\mu\langle w_0\rangle}(K_iL_i^{-1})=1 \quad\forall i\in\I.
\end{align}
Indeed, this follows as Example \ref{example:L one dimensional} using \eqref{eq:Amu twisted} and \eqref{eq:Verma twisted}.
\end{example}

\subsection{Highest weight theory} 

The proofs of the next results run as in \cite{AJS}.

\begin{lemma}[{\cite[Lemma 4.10]{AJS}}]\label{le:iso Z and Z dual}
For all $\mu\in\Z^\I$, we have that $Z_\bA^{w_0}(\mu\langle w_0\rangle)\simeq Z_\bA(\mu)^\tau$ and $Z_\bA(\mu+\beta_{top}^\fq)\simeq Z_\bA^{w_0}(\mu)^\tau$. \qed
\end{lemma}

\begin{proposition}[{\cite[Propositions 4.11 and 4.12]{AJS}}]\label{prop:ext Z}
Let $\lambda,\mu\in\Z^\theta$. Then 
\begin{align*}
\Ext^n_{\cC_\bA^\fq}\left(Z_\bA(\lambda),Z_\bA^{w_0}(\mu)\right)=\begin{cases}
\bA,&\mbox{if $n=0$ and $\mu=\lambda\langle w_0\rangle$}\\
0,&\mbox{otherwise.}
\end{cases}
\end{align*}
\qed
\end{proposition}

A by-product of the above is that the category $\cC_\bk^\fq$ satisfies all the axioms of the definition of highest weight category in \cite[\S3.2]{BGS} except for the axiom (2) since $\cC_\bk^\fq$ has infinitely many simple objects.

\begin{theorem}
$\cC_\bk^\fq$ is a highest weight category with infinitely many simple modules $L_\bk(\lambda)$, $\lambda\in\Z^\I$. The standard and costandard modules are $Z_\bk(\lambda)$ and $Z_\bk(\lambda)^\tau$, $\lambda\in\Z^\I$. \qed
\end{theorem}

Another interesting consequence is the so-called BGG Reciprocity \cite[Proposition 4.15]{AJS}. Let $P_\bk(\lambda)\in\cC^\fq_\bk$ be the projective cover of $L_\bk(\lambda)$, $\lambda\in\Z^\I$. Recall $P_\bk(\lambda)$ admits a $Z$-filtration \cite[Lemma 2.16]{AJS}.  Given $\mu\in\Z^\I$, we denote $[P_\bk(\lambda):Z_\bk(\mu)]$ the number of subquotients isomorphic to $Z_\bk(\mu)$ in a $Z$-filtration of $P_\bk(\lambda)$. 

\begin{theorem}[BGG Reciprocity]\label{teo:BGG}
Let $\lambda,\mu\in\Z^\I$. Then
\begin{align*}
[P_\bk(\lambda):Z_\bk(\mu)]=[Z_\bk(\mu):L_\bk(\lambda)].
\end{align*}
\end{theorem}

\begin{proof}
We know that $[Z_\bk(\mu):L_\bk(\lambda)]=\dim\Hom_{\cC_\bk^\fq}(P_\bk(\lambda),Z_\bk(\mu))$ \cite[4.15 $(2)$]{AJS}. Using Proposition \ref{prop:ext Z} we can deduce that $[P_\bk(\lambda):Z_\bk(\mu)]=\dim\Hom_{\cC_\bk^\fq}\left(P_\bk(\lambda),Z_\bk^{w_0}(\lambda\langle w_0\rangle)\right)$ which is equal to $[Z_\bk^{w_0}(\lambda\langle w_0\rangle):L_\bk(\lambda)]$. Since $\ch\, Z_\bk(\lambda)=\ch\, Z_\bk^{w_0}(\lambda\langle w_0\rangle)$, Lemma \ref{lem:ch and composition} implies the equality in the statement.
\end{proof}

\section{Morphisms between twisted Verma modules}\label{sec:morphisms}

We keep the notation of the previous section. Here we construct generators of the Hom spaces between twisted Verma modules. Recall that they are $\bA$-free of rank one by Proposition \ref{prop:Hom Zx y Zw}. We will proceed in an inductive fashion starting from morphisms between a sort of parabolic Verma modules.

\subsection{Parabolic Verma modules}

We fix $i\in\I$ and write $\sigma_i=1^\fq\sigma_i$. We also fix $\mu\in\Z^\I$. To shorten notation,  we write $\mu'=\mu\langle\sigma_i\rangle=\mu-(b^\fq(\alpha_i)-1)\alpha_i$.

Recall $P_\fq(\alpha_i)$ from \eqref{eq:P alpha i}-\eqref{eq:P alpha i otra}.   By construction $P_\fq(\alpha_i)$ has a triangular decomposition satisfying \eqref{eq:property 0}-\eqref{property D}. We denote ${}_{\alpha_i}\cC_\bA^\fq$ the corresponding AJS category. The {\it parabolic Verma module} and the {\it parabolic $\sigma_i$-Verma module} are
\begin{align}
\Psi_\bA(\mu)=P_\fq(\alpha_i)\ot_{U_\fq^0U_\fq^+}\bA^\mu
\quad\mbox{and}\quad
\Psi'_\bA(\mu')=P_\fq(\alpha_i)\ot_{U^0T_i(U^+_{\sigma_i^{*}\fq})}\bA^{\mu'},
\end{align}
respectively. These are objects of ${}_{\alpha_i}\cC_\bA^\fq$.

Clearly, the elements $F_i^t\m{\mu}$, $0\leq t< b^\fq(\alpha_i)$, form a basis of $\Psi_\bA(\mu)$. By \eqref{eq: E Fn}, the action of $P_\fq(\alpha_i)$ is given by
\begin{align}\label{eq: E Ft mu}
E_i\cdot F_i^t\m{\mu}&=\pi\tilde{\mu}[\alpha_i;t]\,F_i^{t-1}\m{\mu},
\end{align}
$E_j\cdot F_i^t\m{\mu}=0$ for all $j\in\I\setminus\{i\}$ and $F_i\cdot F_i^t\m{\mu}=F_i^{t+1}\m{\mu}$. The weight of $F_i^t\m{\mu}$ is $\mu-t\alpha_i$.

In turn, the elements $E_i^t\m{\mu'}$, $0\leq t<b^{\fq}(\alpha_i)$, form a basis of $\Psi'_A(\mu')$. By \eqref{eq: E Fn},
\begin{align}\label{eq: F Et mu}
F_i\cdot E_i^t\m{\mu'}&=\pi\widetilde{\mu'}[t;\alpha_i]\,E_i^{t-1}\m{\mu'}.
\end{align}
These are vectors of weights $\mu'+t\alpha_i=\mu-(b^\fq(\alpha_i)-1-t)\alpha_i$, respectively. This implies that $E_j\cdot E_i^{t}\m{\mu'}=0$ for all $j\in\I\setminus\{i\}$. Of course, $E_i\cdot E_i^t\m{\mu'}=E_i^{t+1}\m{\mu'}$.

Therefore there exists in ${}_{\alpha_i}\cC_\bA^\fq$ a morphism 
\begin{align}\label{eq:varphi i}
f_i:\Psi_\bA(\mu)\longrightarrow \Psi'_\bA(\mu')
\end{align}
such that $f_i(\m{\mu})= E_i^{b^\fq(\alpha_i)-1}\m{\mu'}$.  Indeed, this is the morphism induced by the fact that $E_j\cdot E_i^{b^\fq(\alpha_i)-1}\m{\mu'}=0$ for all $j\in\I$. Moreover, any morphism from $\Psi_\bA(\mu)$ to $\Psi'_\bA(\mu')$ is a $\bA$-multiple of $\fip$ since the weight spaces are of rank one over $\bA$. We can compute explicitly this morphism as follows:
\begin{align}
\label{eq: varphi i on Fi mu}
\fip(F_i^n\m{\mu})&=F_i^n\cdot E_i^{b^\fq(\alpha_i)-1}\m{\mu'}=\prod_{t=1}^n\pi\widetilde{\mu'}([b^{\fq}(\alpha_i)-t;\alpha_i])\,E_i^{b^\fq(\alpha_i)-1-n}\m{\mu'}
\end{align}
for all $0\leq n<b^\fq(\alpha_i)$. 

Analogously, there exists
\begin{align}
\fip':\Psi'_\bA(\mu')\longrightarrow \Psi_\bA(\mu)
\end{align}
given by
\begin{align}
\label{eq: varphi prima i on Ei mu}
\fip'(E_i^n\m{\mu'})&=E_i^n\cdot F_i^{b^\fq(\alpha_i)-1}\m{\mu}=\prod_{t=1}^n\pi\widetilde{\mu}([\alpha_i;b^{\fq}(\alpha_i)-t])\,F_i^{b^\fq(\alpha_i)-1-n}\m{\mu}.
\end{align}

\begin{lemma}\label{le: unit then iso varphi i}
If $\pi\widetilde{\mu}([\alpha_i;t])$ is a unit in $\bA$ for all $1\leq t\leq b^\fq(\alpha_i)-1$, then $\fip$ and $\fip'$ are isomorphisms.
\end{lemma}

\begin{proof}
The claim for $f_i'$ follows directly from \eqref{eq: varphi prima i on Ei mu} by the hypothesis. By the formula \eqref{eq: varphi i on Fi mu}, $\fip$ is an isomorphism if $\pi\widetilde{\mu'}[b^{\fq}(\alpha_i)-t;\alpha_i]$ is a unit for all $1\leq t\leq b^\fq(\alpha_i)-1$. We have that $q_{ii}^{-t-1}\pi\widetilde{\mu'}(K_iL_i^{-1})=q_{ii}^{1-t}\pi\tilde\mu(K_iL_i^{-1})$ and hence \eqref{eq:[n;beta] = otra} implies that
\begin{align}\label{eq:pi mu' en}
\pi\widetilde{\mu'}[b^{\fq}(\alpha_i)-t;\alpha_i]=(b^\fq(\alpha_i)-t)_{q_{ii}^{-1}}\pi\widetilde{\mu'}(L_i)\left(1-q_{ii}^{1-t}\pi\tilde\mu(K_iL_i^{-1})\right).
\end{align}
Since $(b^\fq(\alpha_i)-t)_{q_{ii}^{-1}}$ and $\pi\widetilde{\mu'}(L_i)$ are units, $\fip$ is an isomorphism by the hypothesis and \eqref{eq:[beta;n] = otra}. 
\end{proof}

We next restrict ourselves to the case of the field $\bk$. To simplify notation, we write 
\begin{align*}
t_i=t^\pi_{\alpha_i}(\mu),
\end{align*}
recall Definition \ref{def:t beta}.  By Lemma \ref{le: unit then iso varphi i}, $\fip$ and $\fip'$ are isomorphisms in ${}_{\alpha_i}\cC_\bk^\fq$ if $t_i=0$.

\begin{lemma}
Suppose $t_i\neq0$. Then
\begin{align*}
\Ker\fip=\Im\fip'=\langle\, F_i^n\m{\mu}\mid n\geq t_i\rangle\quad\mbox{and}\quad
\Im\fip=\Ker\fip'=\langle E_i^n\m{\mu'}\mid n\geq b^\fq(\alpha_i)-t_i\rangle.
\end{align*}
\end{lemma}

\begin{proof}
Using \eqref{eq: varphi i on Fi mu}, \eqref{eq:pi mu' en} and \eqref{eq:[n;beta] = otra}, we see that $\Ker\fip=\langle\, F_i^n\m{\mu}\mid n\geq t_i\rangle$ and $\Im\fip=\langle E_i^n\m{\mu}\mid n\geq b^\fq(\alpha_i)-t_i\rangle$.

For $\fip'$, we first observe that $b^\fq(\alpha_i)-t_i$ is the minimum natural number $s_i$ such that $q_{ii}^{1+s_i}\pi\widetilde{\mu}(K_iL_i^{-1})-1=0
$. Indeed, $q_{ii}^{1+s_i}\pi\widetilde{\mu}(K_iL_i^{-1})-1=0
=q_{ii}^{1-t_i}\pi\tilde\mu(K_i L_i^{-1})-1$ implies
$q_{ii}^{t_i+s_i}=1$, recall \eqref{eq:[beta;n] = otra}. Hence  $b^\fq(\alpha_i)=\ord q_{ii}\leq t_i+s_i$. On the other hand, $\pi\tilde\mu([\alpha_i;b^\fq(\alpha_i)-s_i])=0$ and hence $t_i\leq b^\fq(\alpha_i)-s_i$. Therefore $b^\fq(\alpha_i)=t_i+s_i$. 

Finally, the equalities for $\Im\fip'$ and  $\Ker\fip'$ follow from \eqref{eq: varphi prima i on Ei mu} since $\pi\widetilde{\mu}([\alpha_i;b^{\fq}(\alpha_i)-t])=(b^{\fq}(\alpha_i)-t)_{q_{ii}}\pi\widetilde{\mu}(L_i)(q_{ii}^{1+t}\pi\widetilde{\mu}(K_i L_i^{-1})-1)$.
\end{proof}

As a consequence of the above lemma we see that $F_i^{t_i}\m{\mu}$ is a highest-weight vector, {\it i.e.} $E_j\cdot F_i^{t_i}\m{\mu}=0$ for all $j\in\I$, since the weights of $\Ker\fip$ are $\mu-n\alpha_i$ with $n\geq t_i$. 
Therefore there exists a morphism 
\begin{align}\label{eq:gi}
\gip:\Psi_\bk(\mu-t_i\alpha_i)\longrightarrow\Psi_\bk(\mu) 
\end{align}
in ${}_{\alpha_i}\cC_\bk^\fq$ such that $\m{\mu-t_i\alpha_i}\mapsto F_i^{t_i}\m{\mu}$. Clearly, $\Im\gip=\Ker\fip$.

\begin{lemma}\label{le:long exact sequence for Psi}
There exists a long exact sequence
\begin{align*}
\cdots\Psi_\bk(\mu-(b^\fq(\alpha_i)+t_i)\alpha_i)\longrightarrow\Psi_\bk(\mu-b^\fq(\alpha_i)\alpha_i)\longrightarrow\Psi_\bk(\mu-t_i\alpha_i)\overset{\gip}{\longrightarrow}\Psi_\bk(\mu)\longrightarrow\cdots 
\end{align*}
\end{lemma}

\begin{proof}
The kernel of $\gip$ is generated by $F_i^{b^\fq(\alpha_i)-t_i}\m{\mu-t_i\alpha_i}$ since $\gip(F_i^n\m{\mu-t_i\alpha_i})=F_i^{n+t_i}\m{\mu}$. Hence this generator is a highest-weight vector. Therefore we have a morphism $\Psi_\bk(\mu-b^\fq(\alpha_i)\alpha_i)\longrightarrow\Psi_\bk(\mu-t_i\alpha_i)$ and its kernel is generated by $F_i^{t_i}\m{\mu-b^\fq(\alpha_i)\alpha_i}$. We can repeat the arguments with $\mu-b^\fq(\alpha_i)\alpha_i$ instead of $\mu$ in order to construct the desired long exact sequence.
\end{proof}

\begin{remark}\label{obs:gi'}
In a similar way, we can see that $T_i(E_j)\cdot E_i^{b^\fq(\alpha_i)-t_i}\m{\mu'}=0$ for all $j\in\I$. Therefore there exists a morphism $\gip':\Psi_\bk'(\mu'+(b^\fq(\alpha_i)-t_i)\alpha_i)\longrightarrow\Psi_\bk'(\mu')$ in ${}_{\alpha_i}\cC_\bk^\fq$ such that $\m{\mu'+(b^\fq(\alpha_i)-t_i)\alpha_i}\mapsto E_i^{b^\fq(\alpha_i)-t_i}\m{\mu'}$ and $\Im\gip'=\Ker\fip'$.
\end{remark}

\subsection{Morphisms between Verma modules twisted by a simple reflection}\label{subsec:twisting by a reflection}

Here we lift the morphisms of the above subsection to morphisms between actual Verma modules. We continue with the same fixed elements $i\in\I$ and $\mu\in\Z^\I$. Recall $t_i=t^\pi_{\alpha_i}(\mu)$.

We can construct the Verma modules inducing from the parabolic ones. Namely, we have the next isomorphisms in $\cC_\bA^\fq$:
\begin{align}
\label{eq: ZAmu iso}
Z_\bA(\mu)&\simeq U_\fq\ot_{P_\fq(\alpha_i)}(P_\fq(\alpha_i)\ot_{U_\fq^0U_\fq^+}\bA^\mu)\simeq U_\fq\ot_{P_\fq(\alpha_i)}\Psi_\bA(\mu) 
\quad\mbox{and}\quad\\
\label{eq: ZAmu sigma i iso}
Z^{\sigma_i}_\bA(\mu)&\simeq U_\fq\ot_{P_\fq(\alpha_i)}(P_\fq(\alpha_i)\ot_{U_\fq^0T_i(U^+_{\sigma_i^*\fq})}\bA^\mu)\simeq U_\fq\ot_{P_\fq(\alpha_i)}\Psi'_\bA(\mu).
\end{align}
These isomorphisms allow us to lift $\fip$ and $\fip'$ to morphisms in $\cC^\fq_\bA$.

\begin{lemma}

The morphisms
\begin{align}\label{eq: varphi}
\varphi=1\ot\fip:Z_\bA(\mu)\longrightarrow Z^{\sigma_i}_\bA(\mu\langle\sigma_i\rangle)
\quad\mbox{and}\quad
\varphi'=1\ot\fip':Z^{\sigma_i}_\bA(\mu\langle\sigma_i\rangle)\longrightarrow Z_\bA(\mu)
\end{align}
are generators of the respective Hom spaces as $\bA$-modules. Also,
\begin{align}\label{eq: ker varphi}
\Ker\varphi\simeq U_\fq\ot_{P_\fq(\alpha_i)}\Ker\fip
\quad\mbox{and}\quad
\Ker\varphi'\simeq U_\fq\ot_{P_\fq(\alpha_i)}\Ker\fip'.
\end{align}
\end{lemma}

\begin{proof}
On the space of weight $\mu$, $\varphi$ is an isomorphism by construction. Then $\varphi$ is a generator of $\Hom_{\cC_\bA^\fq}(Z_\bA(\mu),Z^{\sigma_i}_\bA(\mu\langle\sigma_i\rangle))$ by Proposition \ref{prop:Hom Zx y Zw}. By the PBW basis \eqref{eq:PBW basis}, $U_\fq$ is free over $P_\fq(\alpha_i)$ and hence $\Ker\varphi\simeq U_\fq\ot_{P_\fq(\alpha_i)}\Ker\fip$. The proof for $\varphi'$ is analogous. 
\end{proof}

From the above considerations we can extend immediately to $\varphi$ and $\varphi'$ Lemmas \ref{le: unit then iso varphi i}-\ref{le:long exact sequence for Psi}, cf. \cite[Lemmas 5.8 and 5.9]{AJS}.

\begin{lemma}\label{le:varphi iso}
Assume that $\pi\tilde\mu([\alpha_i;t])$ is a unit in $\bA$ for all $1\leq t\leq b^\fq(\alpha_i)-1$. Then $\varphi$ and $\varphi'$ are isomorphisms in $\cC_\bA^\fq$. 
\qed
\end{lemma}

\begin{lemma}
If $t_i=0$, then
$L_\bk(\mu)\simeq L_\bk^{\sigma_i}(\mu\langle\sigma_i\rangle)$ in $\cC_\bk^\fq$. \qed 
\end{lemma}

\begin{lemma}\label{le:varphi no iso}
Suppose $t_i\neq0$. Then
\begin{align*}
\Ker\varphi=\Im\varphi'=U_\fq\cdot F_i^{t_i}\m{\mu}\quad\mbox{and}\quad
\Im\varphi=\Ker\varphi'=U_\fq\cdot E_i^{b^\fq(\alpha_i)-t_i}\m{\mu}.
\end{align*}
\qed
\end{lemma}

Moreover, if $t_i\neq0$, we have that
\begin{align}\label{eq:ch Ker varphi}
\ch\Ker\varphi=e^\mu\left(e^{-t_i\alpha_i}+\cdots+e^{(1-b^\fq(\alpha_i))\alpha_i}\right) 
\prod_{\gamma\in \Delta_+^\fq\setminus\{\alpha_i\}}\frac{1-e^{-b^\fq(\gamma)\gamma}}{1-e^{-\gamma}}.
\end{align}
because of the PBW basis \eqref{eq:PBW basis} and \eqref{eq: ker varphi}.

The morphism $\psi$ below is induced by $g_i$ of \eqref{eq:gi}.

\begin{lemma}\label{le:psi} 
Suppose $t_i\neq0$. Then $F_i^{t_i}\m{\mu}$ is a highest-weight vector of $Z_\bk(\mu)$ and hence there exists a morphism 
\begin{align*}
\psi:Z_\bk(\mu-t_i\alpha_i)\longrightarrow Z_\bk(\mu) 
\end{align*}
in $\cC_\bk^\fq$ such that $\m{\mu-t_i\alpha_i}\mapsto  F_i^{t_i}\m{\mu}$. Moreover, $\Im\psi=\Ker\varphi$ and there is a long exact sequence
\begin{align*}
\cdots Z_\bk(\mu-(b^\fq(\alpha_i)+t_i)\alpha_i)\longrightarrow Z_\bk(\mu-b^\fq(\alpha_i)\alpha_i)\longrightarrow Z_\bk(\mu-t_i\alpha_i)\overset{\psi}{\longrightarrow} Z_\bk(\mu)\longrightarrow\cdots 
\end{align*}
\qed
\end{lemma}

\begin{remark}
There is a morphism $\psi':Z_\bk^{\sigma_i}(\mu-(t_i-1)\alpha_i)\longrightarrow Z_\bk^{\sigma_i}(\mu\langle\sigma_i\rangle)$ in $\cC_\bk^\fq$, induced by $\psi'$ of Remark \ref{obs:gi'}, such that $\m{\mu-(t_i-1)\alpha_i}\mapsto E_i^{b^\fq(\alpha_i)-t_i}\m{\mu\langle\sigma_i\rangle}$ and $\Im\psi'=\Ker\varphi'$.
\end{remark}

From the long exact sequence, we see that
\begin{align*}
\ch(\Ker\varphi)=\sum_{\ell\geq0}\ch Z_\bk(\mu-(\ell b^\fq(\alpha_i)-t_i)\alpha_i)-\sum_{\ell\geq1}\ch Z_\bk(\mu-\ell b^\fq(\alpha_i)\alpha_i).
\end{align*}

As another consequence we obtain the next isomorphisms between simple modules.

\begin{proposition}\label{prop:L iso L sigma}
Suppose $t_i\neq0$. In $\cC_\bk^\fq$, it holds that
\begin{align*}
L_\bk(\mu)\simeq L^{\sigma_i}_\bk(\mu-(t_i-1)\alpha_i)\quad\mbox{and}\quad
L_\bk(\mu-t_i\alpha_i)\simeq L_\bk^{\sigma_i}(\mu\langle\sigma_i\rangle).
\end{align*}
 \end{proposition}

\begin{proof}
Since $\Im\varphi=\Im\psi'$ and $\Im\psi=\Im\varphi'$, the respective domains must to have isomorphic heads, which are the isomorphisms in the statement.
\end{proof}

\begin{remark}
Maps similar to $\varphi$ were constructed in \cite[Lemma 5.8]{HY}. Those maps are just linear morphism between Verma modules in different categories. By considering the $\sigma_i$-Verma module, we obtain morphisms between objects in the same category.
\end{remark}

\subsection{Twisting the categories}\label{subsec:twisting the categories} 
Let $w:\fq\longrightarrow w^{*}\fq$ be a morphism in $\cW$ and $w=\sigma_{i_k}\cdots\sigma_{i_1}1^\fq$ a reduced expression. Fix the Lusztig isomorphism
\begin{align*}
T_w=T_{i_1}\cdots T_{i_k}:U_\fq\longrightarrow U_{w^{*}\fq}.
\end{align*}

Recall the $U_\fq^0$-algebra $\bA[w]$ from \eqref{eq:pi[w]} with structural map
\begin{align}
\pi[w]=\pi\circ T_w^{-1}{}_{|U_\fq^0}:U_\fq^0\longrightarrow\bA.
\end{align}
This does not depend on the presentation of $w$ by \eqref{eq:Tw restricted to U0}. Also, $\bA[w][x]=\bA[xw]$ for $x:w^*\fq\longrightarrow x^*(w^*\fq)=(wx)^*\fq$.

We have equivalence of categories ${}^{w}F^{\fq}_{\bA}:\cC_\bA^\fq\longrightarrow\cC^{w^{*}\fq}_{\bA[w]}$ given by: if $M\in\cC_\bA^\fq$, then $F_w(M)=M[w]$ is an object of $\cC^{w^{*}\fq}_{\bA[w]}$ with the action of $U_{w^*\fq}$ twisted by $T_w^{-1}$, that is
\begin{align*}
u\cdot_{T_w^{-1}} m=T_w^{-1}(u)m \quad\forall m\in M[w],\,u\in U_{w^*\fq}.
\end{align*}
Indeed $M[w]$ satisfies \eqref{eq: cCA is U graded}, since 
\begin{align*}
(U_{w^*\fq})_\alpha\cdot_{T_w^{-1}} M[w]_\mu
&\overset{\eqref{eq:Tw U alpha es U w alpha}}{=}(U_\fq)_{w^{-1}\alpha}\, M_{w^{-1}\mu}\subset M_{w^{-1}\alpha+w^{-1}\mu}=M_{w^{-1}(\alpha+\mu)}=M[w]_{\alpha+\mu}.
\end{align*}
It also satisfies \eqref{eq:compatibilidad sm ma}, since for $m\in M[w]_\mu=M_{w^{-1}\mu}$ and $s\in U^0_{w^*\fq}$ we have that
\begin{align*}
s\cdot_{T_w^{-1}}m=T_w^{-1}(s)m=m(\pi\circ\widetilde{w^{-1}(\mu)}\circ T_w^{-1})(s)\overset{\eqref{eq: mu conjugado por Tw}}{=}
m(\pi\circ T_w^{-1}\circ\widetilde\mu)(s)=
m\pi[w](\widetilde\mu(s))
\end{align*}
Clearly, ${}^{w}F^{\fq}_{\bA}$ depends on the reduced expression of $w$. Its inverse ${}^{w}G^{\fq}_{\bA}:\cC^{w^{*}\fq}_{\bA[w]}\longrightarrow\cC_\bA^\fq$ is given by ${}^{w}G^{\fq}_{\bA}(M)=M[w^{-1}]$ with the action of $U_\fq$ twisted by $T_{w}$. When no confusion can arise, we will write simply $M[w]$ and $M[w^{-1}]$ instead of ${}^{w}F^{\fq}_{\bA}(M)$ and ${}^{w}G^{\fq}_{\bA}(M)$, respectively.

The Verma modules of both categories are related as follows.

\begin{lemma}\label{le: Zx mu w isomorphic to}
Let $x\in{}^\fq\cW$, $w\in\cW^\fq$ and $\mu\in\Z^\I$. Then 
\begin{align}\label{eq:verma twisted}
Z_{\bA}^{x}(\mu)[w]\simeq Z_{\bA[w]}^{wx}(w\mu)
\end{align}
in the category $\cC_{\bA[w]}^{w^*\fq}$. 
Moreover, for the field $\bk$, it holds that
\begin{align}\label{eq:L [w] iso}
L_\bk^x(\mu)[w]\simeq L_{\bk[w]}^{wx}(w\mu). 
\end{align}
 
\end{lemma}

\begin{proof}
We can repeat the arguments in  \cite[4.4$(2)$]{AJS} but we must be thorough with the categories where we are considering the Verma modules. 

We first observe that $Z_\bA^x(\mu)$ has not weights of the form $\mu+x\beta$ for all $\beta\in \Delta_+^{x^{-*}\fq}$ because $\ch\,Z_\bA^x(\mu)=e^{\mu}\ch\,T_x(U^-_{x^{-*}\fq})$. Therefore $Z_\bA^x(\mu)[w]$ has not weights of the form $w\mu+wx\beta$ for all $\beta\in\Delta_+^{x^{-*}\fq}$. Also, the space of weight $w\mu$ of $Z_\bA^x(\mu)[w]$ is $\bA$-spanned by $v_1=\m{\mu}_\bA$. In particular, we have that $T_{wx}(U^+_{x^{-*}\fq})$ annuls $v_1$.

On the other hand, we are thinking of $Z_{\bA[w]}^{wx}(w\mu)$ as an object in $\cC_{\bA[w]}^{w^*\fq}$. Thus, it is constructed using the triangular decomposition of $U_{w^{*}\fq}$ induced by $T_{wx}:U_{x^{-*}\fq}\rightarrow U_{w^*\fq}$. This means that its generator $v_2=\m{w\mu}_{\bA[w]}$ is annuled by $T_{wx}(U^+_{x^{-*}\fq})$. Therefore there exists a morphism $f:Z_{\bA[w]}^{wx}(w\mu)\longrightarrow Z_{\bA}^{x}(\mu)[w]$ given by $v_2\mapsto v_1$.

In a similar fashion, we get a morphism $Z_\bA^x(\mu)\longrightarrow Z_{\bA[w]}^{wx}(w\mu)[w^{-1}]$ in $\cC_\bA^\fq$ such that $v_1\mapsto v_2$. Then, we can transform it into a morphism $g:Z_\bA^x(\mu)[w]\longrightarrow Z_{\bA[w]}^{wx}(w\mu)$ in $\cC_{\bA[w]}^{w^*\fq}$ such that $v_1\mapsto v_2$. Clearly, $g\circ f=\id$ and $g[w^{-1}]\circ f[w^{-1}]=\id$. This proves the isomorphisms between the Verma modules.   Therefore, in the case of the field $\bk$, we deduce that their heads also are isomorphic.
\end{proof}

\subsection{Morphisms between twisted Verma modules} 

Here we extend the results of \S\ref{subsec:twisting by a reflection} to any morphism in the Weyl groupoid using the functors of \S\ref{subsec:twisting the categories} and following the ideas of \cite[\S5.10]{AJS}.

We fix $w\in{}^\fq\cW$ and $\beta\in\Delta^\fq$ such that $\beta=w\alpha_i$ for fixed $\alpha_i\in\Pi^{w^{-*}\fq}$ and $i\in\I$. We also fix  $\mu\in\Z^\I$. We set $t_\beta=t_\beta^\pi(\mu)$, recall Definition \ref{def:t beta}. We will use the functor ${}^wF_{\bA[w^{-1}]}^{w^{-*}\fq}:\cC_{\bA[w^{-1}]}^{w^{-*}\fq}\longrightarrow\cC_\bA^\fq$. For abbreviation, we will denote $M[w]$ and $f[w]$ the image of objects and morphisms through this functor.

Using ${}^wF_{\bA[w^{-1}]}^{w^{-*}\fq}$, Lemma \ref{le: Zx mu w isomorphic to} implies that
\begin{align}\label{eq:w Zx mu w isomorphic to}
\begin{split}
Z_{\bA[w^{-1}]}(w^{-1}\mu)[w]&\simeq Z_{\bA}^{w}(\mu) 
\quad\mbox{and}\\
&Z_{\bA[w^{-1}]}^{\sigma_i}((w^{-1}\mu)')[w]\simeq Z_{\bA}^{w\sigma_i}(\mu-(b^\fq(\beta)-1)\beta) 
\end{split}
\end{align}
for $(w^{-1}\mu)'=(w^{-1}\mu)\langle\sigma_i\rangle=w^{-1}\mu-(b^{w^{-*}\fq}(\alpha_i)-1)\alpha_i$; for the first isomorphism take $x=1^{w^{-*}\fq}$ and for the second one $x=\sigma_i^{(\sigma_iw)^{-*}\fq}$.

On the other hand, let $\varphi$ and $\varphi'$ be the morphisms in the category $\cC_{\bA[w^{-1}]}^{w^{-*}\fq}$ between $Z_{\bA[w^{-1}]}(w^{-1}\mu)$ and $Z_{\bA[w^{-1}]}^{\sigma_i}((w^{-1}\mu)')$ given in \eqref{eq: varphi}. We apply to them the functor ${}^wF_{\bA[w^{-1}]}^{w^{-*}\fq}$ and obtain the morphisms
\begin{align}\label{eq: varphi Zw mu Zws mu prima}
\begin{split}
\varphi[w]&:Z_\bA^w(\mu)\longrightarrow Z^{w\sigma_i}_\bA(\mu-(b^\fq(\beta)-1)\beta)
\quad\mbox{and}\\
\noalign{\smallskip}
&\varphi'[w]:Z^{w\sigma_i}_\bA(\mu-(b^\fq(\beta)-1)\beta)\longrightarrow Z_\bA^w(\mu)\quad\mbox{in $\cC_\bA^\fq$}. 
\end{split}
\end{align}

We are ready to extend the results from \S\ref{subsec:twisting by a reflection}.

\begin{lemma}\label{le:t beta unit}
If $\pi\tilde\mu([\beta;t])$ is a unit for all $1\leq t\leq b^\fq(\beta)-1$, then $\varphi[w]$ and $\varphi'[w]$ are isomorphisms in $\cC_\bA^\fq$. 
\end{lemma}

\begin{proof}
By Lemma \ref{le:varphi iso}, $\varphi$ and $\varphi'$ are isomorphisms if $\pi[w^{-1}]\widetilde{w^{-1}\mu}([\alpha_i;t])$ is a unit for all $1\leq t\leq b^{w^{-*}\fq}(\alpha_i)-1=b^{\fq}(\beta)-1$. As $\pi[w^{-1}]\widetilde{w^{-1}\mu}([\alpha_i;t])=\pi\tilde\mu([\beta;t])$ by \eqref{eq:pi w mu beta}, the lemma follows from the hypothesis.
\end{proof}

In the case of the field $\bk$, we immediately deduce an isomorphism between the heads of the Verma modules.
\begin{lemma}\label{le: t beta =0}
If $t_\beta=0$, then
$
L_\bk^w(\mu)\simeq L_\bk^{w\sigma_i}(\mu-(b^\fq(\beta)-1)\beta)$ in $\cC_\bk^\fq$. \qed 
\end{lemma}

Suppose now $t_\beta\neq0$. Since $t_{\alpha_i}^{\pi[w^{-1}]}(w^{-1}\mu)=t_\beta$ by \eqref{eq: t w alpha i}, we can consider the morphisms $\psi$ and $\psi'$ between $Z_{\bk[w^{-1}]}(w^{-1}\mu)$ and $Z_{\bk[w^{-1}]}^{\sigma_i}((w^{-1}\mu)')$ in the category $\cC_{\bk[w^{-1}]}^{w^{-*}\fq}$ given by Lemma \ref{le:psi}. By applying the functor ${}^{w}F_{\bk[w^{-1}]}^{w^{-*}\fq}$ to this lemma, we obtain the following.

\begin{lemma}\label{le:psi w}
Suppose $t_\beta\neq0$. Then the morphisms
\begin{align*}
\psi[w]:Z_\bk^w(\mu-t_\beta\beta)&\longrightarrow Z_\bk^w(\mu)
\quad\mbox{and}\\
&\psi'[w]:Z_\bk^{w\sigma_i}(\mu-(t_\beta-1)\beta)\longrightarrow Z_\bk^{w\sigma_i}(\mu-(b^\fq(\beta)-1)\beta)
\end{align*}
in $\cC_\bk^\fq$ satisfy that
\begin{align*}
\Ker\varphi[w]=\Im\varphi'[w]=\Im\psi[w]\quad\mbox{and}\quad
\Im\varphi[w]=\Ker\varphi'[w]=\Im\psi'[w].
\end{align*}
We also have a long exact sequence
\begin{align*}
\cdots Z^w_\bk(\mu-(b^\fq(\beta)+t_\beta)\beta)\longrightarrow Z^w_\bk(\mu-b^\fq(\beta)\beta)\longrightarrow Z^w_\bk(\mu-t_\beta\beta)\longrightarrow Z^w_\bk(\mu)\longrightarrow\cdots 
\end{align*}
\qed
\end{lemma}

Similar to Proposition \ref{prop:L iso L sigma}, we deduce an isomorphism between the heads of the $w$-Verma modules above.

\begin{proposition}\label{prop:L w iso L w sigma}
Suppose $t_\beta\neq0$. In $\cC_\bk^\fq$, it holds that
\begin{align*}
L_{\bk}^w(\mu)\simeq L_{\bk}^{w\sigma_i}(\mu-(t_\beta-1)\beta)
\quad\mbox{and}\quad
L_{\bk}^{w}(\mu-t_\beta\beta)\simeq L_{\bk}^{w\sigma_i}(\mu-(b^\fq(\beta)-1)\beta).
\end{align*}
\qed
\end{proposition}

\begin{remark}
Using iteratively the prior lemma, we can calculate $\mu_w\in\Z^\I$ such that $L_{\bk}^w(\mu)=L_{\bk}(\mu_w)$, recall \eqref{eq:Lw L mu w}.
\end{remark}

Let us make some comments about the kernel of $\varphi[w]$ in the case of the field $\bk$. First, using \eqref{eq:ch Ker varphi} and $\Ker(\varphi[w])=\Ker(\varphi)[w]$, we have that 
\begin{align}\notag
\ch\Ker(\varphi[w])=e^\mu\left(e^{-t_\beta\beta}+\cdots+\right.&\left.e^{(1-b^\fq(\beta))\beta}\right)\times\\
\label{eq:ch Ker varphi w}
\prod_{\gamma\in\Delta^\fq_+\setminus\{\beta\}:w^{-1}\gamma\in \Delta^{w^{-*}\fq}_+}&\left(1+e^{-\gamma}+\cdots+e^{(1-b^\fq(\gamma))\gamma}\right)\times\\
\notag
&\prod_{\gamma\in\Delta^\fq_+\setminus\{\beta\}:w^{-1}\gamma \in \Delta^{w^{-*}\fq}_-}\left(1+e^{\gamma}+\cdots+e^{(b^\fq(\gamma)-1)\gamma}\right).
\end{align}
Then, if $\beta=w\alpha_i\in\Delta^\fq_+$, it follows that
\begin{align}\label{eq:mu is not a weight of Ker}
\mu-(w(\varrho^{w^{-*}\fq})-\varrho^\fq)\mbox{ is not a weight of }\Ker(\varphi[w])
\end{align}
since $\mu-(w(\varrho^{w^{-*}\fq})-\varrho^\fq)=\mu+\sum_{\gamma\in \Delta_+^\fq\,:\,w^{-1}\gamma\in \Delta_-^{w^{-*}\fq }}(b^\fq(\gamma)-1)\gamma$ by \eqref{eq:w varrho - varrho}. Lastly, we claim that 
\begin{align}
T_w(F_i)^{t_\beta}\m{\mu}_\bk\mbox{ is a  $U_\fq$-generator of }\Ker(\varphi[w])
\end{align}
where $T_w:U_{w^{-*}\fq}\longrightarrow U_\fq$ is a Lusztig isomorphism associated to $w$. In fact, $\psi(\m{w^{-1}\mu-t_\beta\alpha_i}_{\bk[w^{-1}]})=F_i^{t_\beta}\m{w^{-1}\mu}_{\bk[w^{-1}]}$ by Lemma \ref{le:psi}. From the proof of Lemma \ref{le: Zx mu w isomorphic to}, we see that $\psi[w](\m{\mu-t_\beta\beta}_\bk)=g(F_i^{t_\beta}\m{w^{-1}\mu}_{\bk[w^{-1}]})=g(T_w(F_i^{t_\beta})\cdot_{T_w^{-1}}\m{w^{-1}\mu}_{\bk[w^{-1}]})=T_w(F_i^{t_\beta})\m{\mu}_{\bk}$.

\subsubsection{A generator}\label{sss:a generator}

We end this subsection by constructing a generator of the Hom spaces between twisted Verma modules as in Proposition \ref{prop:Hom Zx y Zw}. We follow the strategy in \cite[\S5.13]{AJS}.

We fix $x:x^{-*}\fq\rightarrow\fq$ and a reduced expression $x^{-1}w=1^{x^{-*}\fq}\sigma_{i_1}\cdots\sigma_{i_r}$. We set 
\begin{align*}
x_s=1^\fq  x\sigma_{i_1}\cdots\sigma_{i_{s-1}}                                                                                                                                                                                                                                          \end{align*}
for $1\leq s\leq r+1$. Then
\begin{align*}
\mu\langle x_{s+1}\rangle=\mu\langle x_{s}\rangle-(b^{\fq}(x_s\alpha_{i_s})-1)x_s\alpha_{i_s}
\end{align*}
for all $1\leq s\leq r$ by \eqref{eq:mu w s}. Notice $x_1=x$ and $x_{r+1}=w$. We set $\varphi_s=\varphi[x_s]$ the morphism in $\cC_\bA^\fq$ given by \eqref{eq: varphi Zw mu Zws mu prima} for $\mu\langle x_s\rangle$. Explicitly,
\begin{align*}
\varphi_s:Z_\bA^{x_s}(\mu\langle x_s\rangle)\longrightarrow 
Z_\bA^{x_{s+1}}(\mu\langle x_{s+1}\rangle)
\end{align*}
for all $1\leq s\leq r$. We get a morphism
\begin{align*}
\varphi_{r}\cdots\varphi_1:Z_\bA^{x}(\mu\langle x\rangle)\longrightarrow 
Z_\bA^{w}(\mu\langle w\rangle).
\end{align*}

\begin{proposition}\label{prop:hom generator}
Suppose that either $\pi\widetilde{\mu\langle x_s\rangle}([x_s\alpha_{i_s};t])$ is a unit for all $1\leq t\leq b^{\fq}(x_s\alpha_{i_s})-1$ and $1\leq s\leq r$, or $\bA=\bk$ is a field and $x_s\alpha_{i_s}\in\Delta_+^\fq$ for all $1\leq s\leq r$. Then $\varphi_{r}\cdots\varphi_1$ induces an $\bA$-isomorphism between the spaces of weight $\mu$ and therefore $\varphi_{r}\cdots\varphi_1$ is  an $\bA$-generator of the corresponding $\Hom$ space.
\end{proposition}

\begin{proof}
The spaces of weight $\mu$ are free of rank $1$ over $\bA$ by \eqref{eq: ch Zmu = ch Zw muw} and then the first assertion implies the second one. Also, it is enough to prove it for each $s$. Under the first supposition, $\varphi_s$ is an isomorphism by Lemma \ref{le:t beta unit}. Under the second one, $\mu$ is not a weight of $\Ker\varphi_s$ by \eqref{eq:mu is not a weight of Ker} as $\mu=\mu\langle x_s\rangle-(x_s(\varrho^{x_s^{-*}\fq})-\varrho^\fq)$. Therefore $\varphi_s$ is an isomorphism on the spaces of weight $\mu$.
\end{proof}

We will use the following corollary with $x=1^\fq$ and $w=w_0\in{}^\fq\cW$ the longest element to prove the linkage principle in the next section.

\begin{corollary}\label{cor:a generator}
If $\ell(w)=\ell(x)+\ell(x^{-1}w)$, then $\varphi_{r}\cdots\varphi_1$ is  an $\bk$-generator of the corresponding Hom space.
\end{corollary}

\begin{proof}
We have that $x_{s}\alpha_{i_s}\in\Delta_+^\fq$ by \cite[Lemma 8 $(i)$]{HY08} for all $1\leq s\leq r$.
\end{proof}

\section{The linkage principle}\label{sec:linkage}

We are ready to prove our main results. We follow here the ideas in \cite[\S6]{AJS}. We keep the notation of the above sections and restrict ourselves to the case of the field $\bk$. Recall Definition \ref{def:n beta}

\begin{definition}
Let $\beta\in\Delta_+^\fq$ and $\mu\in\Z^\I$. We set
\begin{align*}
\beta\downarrow\mu=\mu-n_\beta^\pi(\mu)\,\beta.
\end{align*}
We say that $\lambda\in\Z^\I$ is strongly linked to $\mu$ if and only if there exist $\beta_1, ..., \beta_r\in\Delta_+^\fq$ such that $\lambda=\beta_r\downarrow\cdots\beta_1\downarrow\mu$. We denote 
\begin{align*}
{}^{\downarrow}\mu=\left\{\lambda\in\Z^\I\mid\lambda\,\mbox{is strongly linked to}\,\mu\,\mbox{and}\,\mu-\beta_{top}^\fq\leq\lambda\leq\mu
\right\}. 
\end{align*}
Lastly, being linked is the smallest equivalence relation in $\Z^\I$ such that $\lambda$ and $\mu$ are linked if $\lambda$ is strongly linked to $\mu$ or {\it vice versa}. We denote $[\mu]_{\link}$ the equivalence class of $\mu\in\Z^\I$.
\end{definition}

\begin{theorem}\label{teo:strongly linked}
If $L_\bk(\lambda)$ is a composition factor of $Z_\bk(\mu)$, then $\lambda=\mu$ or $L_\bk(\lambda)$ is a composition factor of $Z_\bk(\beta\downarrow\mu)$ for some $\beta\in\Delta_+^\fq$. Moreover, $\lambda\in{}^{\downarrow}\mu$. 
\end{theorem}

\begin{figure}[h]
\begin{tikzpicture}[scale=1,every node/.style={scale=1}]

\node (0) at (-8,0) {$Z_\bk(\mu)$};

\node (s) at (0,0) {$Z^{w_s}_\bk(\mu\langle w_s\rangle)$};

\node (s1) at (4,0) {$Z^{w_{s+1}}_\bk(\mu\langle w_{s+1}\rangle)$};

\node (ks) at (0,-1.5) {$\Ker\varphi_s$};

\node (sp) at (0,-3) {$Z^{w_s}_\bk(\mu\langle w_s\rangle-t_s\beta_s)$};

\node (0p) at (-8,-3) {$Z_\bk(\mu-t_s\beta_s)$};

\draw[->] (0) to node [above,sloped] {\tiny$
\varphi_1$} (-6,0);

\draw[dotted] (-5,0) to (-4,0);

\draw[->] (-3,0) to node [above,sloped] {\tiny$\varphi_{s-1}$} (s);

\draw[->] (s) to node [above,sloped] {\tiny$\varphi_s$} (s1);

\draw[right hook->] (ks) to node [left] {} (s);

\draw[->>] (sp) to node [left] {\tiny$\psi$} (ks);

\draw[->] (0p) to node [above,sloped] {\tiny$\tilde
\varphi_1$} (-6,-3);

\draw[dotted] (-5,-3) to (-4,-3);

\draw[->] (-3,-3) to node [above,sloped] {\tiny$\tilde\varphi_{s-1}$} (sp);

\end{tikzpicture}
\caption{Let
$w_0=1^{\fq}\sigma_{i_1}\cdots\sigma_{i_n}$ be a reduced expression of the longest element in ${}^\fq\cW$. We set $w_s=1^\fq\sigma_{i_1}\cdots\sigma_{i_{s-1}}$ and $\beta_s=w_s\alpha_{i_s}$, $1\leq s\leq n+1$. The morphisms $\varphi_i$ are constructed as in \S\ref{sss:a generator} with $x=1^\fq$ and $w=w_0$. The morphism $\psi$ is given by Lemma \ref{le:psi w} with $t_s=t_{\beta_s}^\pi(\mu\langle w_s\rangle)$. Moreover, $t_s=n_{\beta_s}^\pi(\mu)$ by Lemma \ref{le:n beta = t beta}. Then the compositions $\varphi_s\cdots\varphi_1$ are generators of the corresponding Hom spaces by Corollary \ref{cor:a generator} and $\Phi=\varphi_n\cdots\varphi_1$ satisfies Lemma \ref{le:imagen de Z w0 mu w0 en Z mu}.
The morphisms $\tilde\varphi_i$ are constructed analogously starting from $\mu-t_s\beta_s$ instead of $\mu$.
}
\label{fig:linked}
\end{figure}

\begin{proof}
We use the notation of Figure \ref{fig:linked}. Then $w_0=w_{n+1}$ and $\mu\langle w_{n+1}\rangle=\mu-\beta^\fq_{top}$. Therefore $
\Phi=\varphi_n\cdots\varphi_1:Z_\bk(\mu)\longrightarrow Z^{w_{0}}_\bk(\mu-\beta^\fq_{top})
$ is non-zero by Corollary \ref{cor:a generator} and hence
$
L_\bk(\mu)\simeq\Im\Phi\simeq\soc Z_\bk^{w_0}(\mu-\beta^\fq_{top})
$ by Lemma \ref{le:imagen de Z w0 mu w0 en Z mu}.

Now, suppose $\lambda\neq\mu$. Then $L_\bk(\lambda)$ is a composition factor of $\Ker\Phi$ and hence of $\Ker\varphi_s$ for some $s$. By Lemma \ref{le:psi w}, $\Ker\varphi_s$ is a homomorphic image of $Z^{w_s}_\bk(\mu\langle w_s\rangle-t_s\beta_s)$. Therefore $L_\bk(\lambda)$ is also a composition factor of $Z^{w_s}_\bk(\mu\langle w_s\rangle-t_s\beta_s)$. By Lemma \ref{le:n beta = t beta}, we have that $t_s=n_{\beta_s}^\pi(\mu)$ and hence
\begin{align*}
\mu\langle w_s\rangle-t_s\beta_s=(\mu-t_s\beta_s)\langle w_s\rangle=(\beta_s\downarrow\mu)\langle w_s\rangle.
\end{align*}
Thus, $\ch Z^{w_s}_\bk(\mu\langle w_s\rangle-t_s\beta_s)=\ch Z_\bk(\beta_s\downarrow\mu)$ by \eqref{eq: ch Zmu = ch Zw muw}, and hence $L_\bk(\lambda)$ is also a composition factor of $Z_\bk(\beta_s\downarrow\mu)$. This shows the first part of the statement. For the second one, we repeat the reasoning with $\beta_s\downarrow\mu$ instead of $\mu$, and so on. This procedure will end after a finite number of steps since $\beta\downarrow\mu\leq\mu$ and $\mu-\beta_{top}^\fq\leq\lambda<\mu$.  Hence, there exist $\beta_{s_1}$, ..., $\beta_{s_r}$ such that $\lambda=\beta_{s_r}\downarrow\cdots\beta_{s_1}\downarrow\mu$ as desired. 
\end{proof}

Besides the following particular case, it is not necessarily true that $L_\bk(\lambda)$ is a composition factor of $Z_\bk(\mu)$ if $\lambda\in{}^\downarrow\mu$, see Example \ref{ex:the example a simple}.

\begin{lemma}\label{le:Linkage}
If $\lambda=\beta\downarrow\mu$, then $L_\bk(\lambda)$ is a composition factor of $Z_\bk(\mu)$.
\end{lemma}

\begin{proof}
We can assume we are in the situation of Figure \ref{fig:linked}. That is, $\beta=\beta_s=w_s\alpha_{i_s}$ and $\lambda=\beta_s\downarrow\mu=\mu-t_s\beta_s$. Thus, we have a projection from $Z_\bk^{w_s}(\mu\langle w_s\rangle)$ to $\Ker\varphi_s$. Notice that $\lambda$ is a weight of $\Ker\varphi_s$ by \eqref{eq:ch Ker varphi w}. On the other hand, $\tilde\varphi_{s-1}\cdots\tilde\varphi_1$ induces a $\bk$-isomorphism between the spaces of weight $\lambda$ by Proposition \ref{prop:hom generator}  which is one-dimensional. Hence $\psi\tilde\varphi_{s-1}\cdots\tilde\varphi_1$ is not zero on the space of weight $\lambda$ because $\Im\psi=\Ker\varphi_s$. This implies that $L_\bk(\lambda)$ is a composition factor of $\Ker\varphi_s$ and therefore so is of $Z^{w_s}_\bk(\mu\langle w_s\rangle)$. As $\ch Z^{w_s}_\bk(\mu\langle w_s\rangle)=\ch Z_\bk(\mu)$, the lemma follows.
\end{proof}

In general, we can assert that the strongly linked weights belong to the same block.

\begin{corollary}
Let $\lambda,\mu\in\Z^\I$. Then $\lambda$ and $\mu$ are linked if and only if $L_\bk(\lambda)$ and $L_\bk(\mu)$ belong to the same block.
\end{corollary}

\begin{proof}
We first prove that being linked implies the belonging to the same block. To this end, it is enough to consider the case $\lambda=\beta\downarrow\mu$ for some $\beta\in\Delta_+^\fq$ which follows from Lemma \ref{le:Linkage}.

For the reciprocal, as above, it is enough to consider the existence of a non-trivial extension of the form $0\longrightarrow L_\bk(\lambda)\longrightarrow M\longrightarrow L_\bk(\mu)\longrightarrow0$. Therefore $M$ is a quotient of $Z_\bk(\mu)$ and hence $\lambda\in{}^{\downarrow}\mu$ by Theorem \ref{teo:strongly linked}.
\end{proof}

\subsection{Typical weights}\label{subsec:typical}
For $\beta\in\Delta_+^\fq$ and $\mu\in\Z^\I$, we introduce
\begin{align}\label{eq:P}
\begin{split}
\mathfrak{P}_\bk^\fq(\beta,\mu)&=\prod_{1\leq t<b^\fq(\beta)} \left(q_{\beta}^{t}-\rho^\fq(\beta)\,\pi\widetilde{\mu}(K_{\beta}L_{\beta}^{-1})\right)
\quad\mbox{and}\\
\noalign{\smallskip}
&\mathfrak{P}_\bk^\fq(\mu)=\prod_{\beta\in\Delta_+^\fq}\mathfrak{P}_\bk^\fq(\beta,\mu) 
\end{split}
\end{align}
A weight $\mu$ is called {\it typical} if $\mathfrak{P}_\bk^\fq(\mu)\neq0$. Otherwise, it is called {\it atypical} and the numbers of positive roots for which $\mathfrak{P}_\bk^\fq(\beta,\mu)=0$ is its {\it degree of atypicality}; if it is $\ell$ we say that $\mu$ is $\ell$-atypical. This terminology is borrowed from \cite{kac}, see also \cite{ser,Y}.

\begin{corollary}[{\cite[\S6.3]{AJS}}]\label{cor:Linkage}

Let $\mu\in\Z^\I$. The following are equivalent:
\begin{enumerate}
 \item $\mu$ is typical.
 \item $Z_\bk(\mu)=L_\bk(\mu)$ is simple.
 \item $Z_\bk(\mu)=L_\bk(\mu)$ is projective.
\end{enumerate}

\end{corollary}

\begin{proof}
If $\mu$ is typical, then $L_\bk(\mu)$ is the unique composition factor of $Z_\bk(\mu)$ by Theorem \ref{teo:strongly linked}. Since $[Z_\bk(\mu),L_\bk(\mu)]=1$, it follows that $Z_\bk(\mu)$ is simple. Instead, if $\mathfrak{P}_\bk^\fq(\mu)=0$, then $\Ker\varphi_s\neq0$ for some $s$, cf. Figure \ref{fig:linked}. Since the morphism $\varphi_n\cdots\varphi_1:Z_\bk(\mu)\longrightarrow Z^{w_{0}}_\bk(\mu-\beta^\fq_{top})$ is non trivial, $\Ker\varphi_s\neq Z_\bk(\mu)$ and hence  $Z_\bk(\mu)$ is non simple. This proves that (1) is equivalent to (2). 

If $Z_\bk(\mu)=L_\bk(\mu)$, then $[P_\bk(\mu):Z_\bk(\nu)]=\delta_{\mu,\nu}$  for $\nu\in\Z^\I$ by Theorem \ref{teo:BGG}. Therefore $P_\bk(\mu)=Z_\bk(\mu)=L_\bk(\mu)$ is also projective. This shows that (2) implies (3), and the converse follows by a similar argument.
\end{proof}

\begin{remark}
The equivalence between $(1)$ and $(2)$ was proved before in \cite[Proposition 5.16]{HY}; see also \cite[Remark 6.25]{Y}.
\end{remark}

\subsection{\texorpdfstring{$1$}{1}-atypical weights}

For weights with degree of atypicality $1$ we can compute the character of the associated simple module similar to \cite[\S6.4]{AJS}.

\begin{corollary}
Let $\mu\in\Z^\I$ be a $1$-atypical weight with $\mathfrak{P}_\bk^\fq(\beta,\mu)=0$ for certain $\beta\in\Delta_+^\fq$. Then 
\begin{align*}
\ch L_\bk(\mu)=e^\mu\quad\frac{1-e^{-n^\pi_\beta(\mu)\beta}}{1-e^{-\beta}} 
\prod_{\gamma\in \Delta_+^\fq\setminus\{\beta\}}\frac{1-e^{-b^\fq(\gamma)\gamma}}{1-e^{-\gamma}}.
\end{align*}
Moreover, there exists an exact sequence
\begin{align*}
0\longrightarrow L_\bk(\beta\downarrow\mu)\longrightarrow Z_\bk(\mu)\longrightarrow L_\bk(\mu)\longrightarrow0.
\end{align*}
\end{corollary}

\begin{proof}
We keep the notation of Figure \ref{fig:linked}: $\beta=\beta_s=w_s\alpha_{i_s}$, $\lambda=\beta_s\downarrow\mu=\mu-t_s\beta_s$ and $\Phi=\varphi_n\cdots\varphi_1$. As $\mu$ is $1$-atypical, the morphisms $\varphi_\ell$ are isomorphisms for all $\ell\neq s$ by Lemma \ref{le:t beta unit} and hence $\Im\varphi_s\simeq\Im\Phi\simeq L_\bk(\mu)$, recall Lemma \ref{le:imagen de Z w0 mu w0 en Z mu}. Thus, the character formula is consequence of \eqref{eq:ch Ker varphi w}.

For the existence of the exact sequence, we claim that $L_\bk(\lambda)\simeq\Ker\varphi_s$. Indeed, in the proof of Lemma \ref{le:Linkage}, we saw that $L_\bk(\lambda)$ is a composition factor of $\Ker\varphi_s$. Thus, $\dim L_\bk(\lambda)\leq\dim \Ker\varphi_s= (b^q(\beta)-n_\beta^\pi(\mu))\prod_{\gamma\in \Delta_+^\fq\setminus\{\beta\}}b^\fq(\gamma)$.
On the other hand,  $n_\beta^\pi(\lambda)=b^q(\beta)-n_\beta^\pi(\mu)$ and hence $\dim\Ker\tilde\varphi_s=n_\beta^\pi(\mu)\prod_{\gamma\in \Delta_+^\fq\setminus\{\beta\}}b^\fq(\gamma)$ by \eqref{eq:ch Ker varphi w}. Let $\tilde\Phi=\tilde\varphi_n\cdots\tilde\varphi_1$. Then $L_\bk(\lambda)\simeq\Im\tilde\Phi$ by Lemma \ref{le:imagen de Z w0 mu w0 en Z mu} and therefore $\dim L_\bk(\lambda)=\dim Z_\bk(\lambda)-\dim\Ker\tilde\Phi\geq\dim Z_\bk(\lambda)-\dim\Ker\tilde\varphi_s=(b^q(\beta)-n_\beta^\pi(\mu))\prod_{\gamma\in \Delta_+^\fq\setminus\{\beta\}}b^\fq(\gamma)$. This implies our claim and the corollary is proved.
\end{proof}

\begin{example}\label{ex:the example a simple}
Let $\fq$ be as in Example \ref{ex:the example}. Its positive roots are $\alpha_1$, $\beta=\alpha_1+\alpha_2$ and $\alpha_2$, recall Example \ref{ex:the example pbw}. Let $\pi:U_\fq^0\longrightarrow\ku$ be an algebra map. For $\mu=0$, we have that
\begin{align*}
\mathfrak{P}_\ku^\fq(0)=
\bigl(-1+\pi(K_1L_1^{-1})\bigr)\,
\prod_{t=1}^{N-1}\bigl(q^t-\pi(K_{\beta}L_{\beta}^{-1})\bigr)\,
\bigl(-1+\pi(K_2L_2^{-1})\bigr).
\end{align*}
Suppose $\pi(K_{\beta}L_{\beta}^{-1})=q^t$ for some $1\leq t\leq N-1$ and $\pi(K_1L_1^{-1})\neq1\neq\pi(K_2L_2^{-1})$. Then $\mu=0$ is $1$-atypical, $\beta\downarrow0=-t\beta$ and hence there is an exact sequence of the form
\begin{align*}
0\longrightarrow L_\ku(-t\beta)\longrightarrow Z_\ku(0)\longrightarrow L_\ku(0)\longrightarrow0.
\end{align*}
Moreover, $\ch L_\ku(0)=\left(1+e^{-\alpha_1}\right)\left(1+e^{-\beta}+\cdots+e^{(1-t)\beta}\right)\left(1+e^{-\alpha_2}\right)$.

We observe now that 
$q_{\beta}^{N-t}-\rho^\fq(\beta)\,\pi\widetilde{(-t\beta)}(K_{\beta}L_{\beta}^{-1})=0$
and hence 
\begin{align*}
\beta\downarrow\beta\downarrow0=\beta\downarrow-t\beta=-t\beta-(N-t)\beta=-N\beta=-\beta^\fq_{top}.
\end{align*}
Therefore $-\beta^\fq_{top}\in{}^\downarrow0$ but $L_\ku(-\beta^\fq_{top})$ is not a composition factor of $Z_\ku(0)$.
\end{example}

\subsection{The linkage principle as a dot action}\label{subsec:linkage via dot}

In this subsection, we assume that $\fq$ is of standard type \cite{AA-diag-survey}, this means  the bundles of matrices $\{C^\fp\}_{\fp\in\cX}$ and roots $\{\Delta^\fp\}_{\fp\in\cX}$ are constant. We will see that the operation $\downarrow$ can be carried out as the action of a group
when $\pi:U_\fq^0\longrightarrow\bk$ satisfies $\pi(K_i)=\pi(L_i)=1$ for all $i\in\I$, {\it e.g.} $\pi=\varepsilon$ the counit.

Let us introduce some notation. For $i\in\I$, we define the group homomorphism
\begin{align*}
\langle\alpha^\vee_i,-\rangle:\Z^\I\longrightarrow\Z\quad\mbox{by}\quad\langle\alpha^\vee_i,\alpha_j\rangle=c_{ij}^\fq\quad\forall j\in\I.
\end{align*}
Therefore
\begin{align*}
\sigma_i=\sigma_i^\fp(\mu)=\mu-\langle\alpha^\vee_i,\mu\rangle\,\alpha_i
\end{align*}
for all $\mu\in\Z^\I$ and $\fp\in\cX$, as the bundle of Cartan matrices is constant.

In the next definition we think of the morphisms in the Weyl groupoid just as $\Z$-automorphisms of $\Z^\I$.

\begin{definition}\label{def:s beta}
Let $\beta=w\alpha_i\in\Delta^\fq$ with $w\in{}^\fq\cW$ and $\alpha_i\in\Pi^{w^{-*}\fq}$. We define $s_\beta\in\Aut_\Z(\Z^\I)$ and the group homomorphism $\langle\beta^\vee,-\rangle:\Z^\I\longrightarrow\Z$ as follows
\begin{align*}
s_\beta=w\,\sigma_i\, w^{-1}\quad\mbox{and}\quad\langle\beta^\vee,\mu\rangle=\langle\alpha_i^\vee,w^{-1}\mu\rangle\quad\forall\mu\in\Z^\I.
\end{align*}
\end{definition}

Of course, $s_\beta$ is defined for all roots thanks to \eqref{eq:roots are conjugate to simple}. This definition and the next lemma are in \cite[\S3.2]{AARB} for Cartan roots. The proof runs essentially as in {\it loc. cit.}

\begin{lemma}
Let $\beta\in\Delta^\fq$. Then $s_\beta$ and $\langle\beta^\vee,-\rangle$ are well-defined, that is, they do not depend on $w$ and $\alpha_i$. Moreover, $s_\beta(\beta)=-\beta$ and
\begin{align*}
s_\beta(\mu)=\mu-\langle\beta^\vee,\mu\rangle\,\beta\quad\forall\mu\in\Z^\I.
\end{align*}
\end{lemma}

\begin{proof}
Assume $\beta=w\alpha_i$ for certain $w\in{}^\fq\cW$ and $\alpha_i\in\Pi^{w^{-*}\fq}$. Then
\begin{align*}
s_\beta(\mu)=w\,\sigma_i(w^{-1}\mu)=w(w^{-1}\mu-\langle\alpha^\vee_i,w^{-1}\mu\rangle\,\alpha_i)=\mu-\langle\beta^\vee,\mu\rangle\,\beta.
\end{align*}
This implies that $s_\beta$ is a reflection in $\operatorname{End}_{\mathbb{Q}}(\mathbb{Q}^\theta)$ in the sense of \cite[Chapitre V \S2.2]{B}. Also, $s_\beta(\Delta^\fq)=\Delta^\fq$ as we are assuming $\fq$ is of standard type. Therefore $s_\beta$ is well-defined, and hence so is $\langle\beta^\vee,-\rangle$, by \cite[Chapitre VI \S1, Lemme 1]{B}. This proves the lemma.
\end{proof}

We recall \cite[Definition 2.6]{Ang-reine}: $i\in\I$ is a Cartan vertex of $\fp\in\cX$ if $\fp(\alpha_i,\alpha_i)^{c_{ij}^\fp}=\fp(\alpha_i,\alpha_j)\fp(\alpha_j,\alpha_i)$ for all $j\in\I$. The set of Cartan roots of $\fq$ is 
\begin{align*}
\Delta_{\car}^\fq=\left\{w(\alpha_i)\mid w:\fp\rightarrow\fq\mbox{ and $i$ is a Cartan vertex of $\fp$}\right\}.
\end{align*}

We introduce a Cartan-type Weyl group
\begin{align*}
\cW_{\car}^\fq=\langle s_\beta\mid\beta\in\Delta_{\car}^\fq\rangle\subset\Aut_\Z(\Z^\I),
\end{align*}
and its affine extension
\begin{align*}
\cW_{\aff}^\fq=\cW_{\car}^\fq\ltimes\Z^\I.
\end{align*}
For $m\in\Z$, we denote $s_{\beta,m}=s_\beta\ltimes mb^\fq(\beta)\beta\in\cW_{\aff}^\fq$ and 
\begin{align*}
\cW_{\link}^\fq=\langle s_{\beta,m}\mid\beta\in\Delta_{\car}^\fq,m\in\Z\rangle\subset\cW^\fq_{\aff}.
\end{align*}
Finally, we define the dot action of $\cW^\fq_{\aff}$ on $\Z^\I$ as
\begin{align*}
(w\gamma)\bullet\mu=w(\mu+\gamma-\varrho^\fq)+\varrho^\fq 
\end{align*}
for all $w\in\cW_{\car}^\fq$ and $\gamma,\mu\in\Z^\I$.

\begin{lemma}\label{le:linkage via dot}
Assume that $\pi:U_\fq^0\longrightarrow\bk$ satifies $\pi(K_j)=\pi(L_j)=1$ for all $j\in\I$. Let $\beta\in\Delta_+^\fq$ be a Cartan root and $\mu\in\Z^\I$. Hence there exists $m\in\Z$ such that
\begin{align*}
\beta\downarrow\mu=s_{\beta,m}\bullet\mu.
\end{align*}
\end{lemma}

\begin{proof}
Let $w:\fp\rightarrow\fq$ and $\alpha_i\in\Pi^\fp$ with $i\in\I$ a Cartan vertex such that $\beta=w\alpha_i$. For abbreviation, we set $n=n_\beta^\pi(\mu)$ and $b=b^\fq(\beta)=\ord q_\beta$. If $1\leq n\leq b-1$, then
\begin{align*}
q_\beta^n=&\rho^\fq(\beta)\pi\tilde\mu(K_\beta L_\beta^{-1})\\
=&\fp(\alpha_i,\alpha_i)\fp(\alpha_i,w^{-1}(\mu\langle w\rangle))\fp(w^{-1}(\mu\langle w\rangle),\alpha_i)\\
=&\fp(\alpha_i,\alpha_i)\fp(\alpha_i,\alpha_i)^{\langle\alpha_i^\vee,w^{-1}(\mu\langle w\rangle)\rangle}\\
=&q_\beta^{\langle\beta^\vee,\mu\langle w\rangle\rangle+1};
\end{align*}
the second equality follows from \eqref{eq:rho = ...} and the assumption on $\pi$; the third one holds because $i$ is a Cartan vertex. Therefore $n\equiv \langle\beta^\vee,\mu\langle w\rangle\rangle+1\mod b$. If $n=0$, then $q_\beta^t\neq q_\beta^{\langle\beta^\vee,\mu\langle w\rangle\rangle+1}$ for all $1\leq t\leq b-1$ and hence $1= q_\beta^{\langle\beta^\vee,\mu\langle w\rangle\rangle+1}$ because $b$ is the order of $q_\beta$. In both cases there exists $k\in\Z$ such that
\begin{align*}
n+kb=\langle\beta^\vee,\mu\langle w\rangle\rangle+1. 
\end{align*}
We claim that $m=1-k$ has the desired property. Indeed, we notice that $\fp(\gamma,\gamma)=\sigma_i^{*}\fp(\gamma,\gamma)$ for all $\gamma\in\Z^\I$ because $i$ is a Cartan vertex, and then $\varrho^\fp=\varrho^{\sigma_i^{*}\fp}$ as $\Delta^\fp=\Delta^{\sigma_i^{*}\fp}$. Hence $\sigma_i(\varrho^{\fp})-\varrho^\fp=\sigma_i(\varrho^{\sigma_i^{*}\fp})-\varrho^\fp=-(b-1)\alpha_i=-\langle\alpha_i^\vee,\varrho^\fp\rangle\alpha_i$. Therefore
$\langle\beta^\vee,w\varrho^\fp\rangle=\langle\alpha_i^\vee,\varrho^\fp\rangle=b-1$. We use this equality in the next computation:
\begin{align*}
s_\beta\bullet\mu=s_\beta(\mu-\varrho^\fq)+\varrho^\fq
=&\mu-\varrho^\fq-\langle\beta^\vee,\mu-\varrho^\fq\rangle\beta+\varrho^\fq\\
=&\mu-\langle\beta^\vee,\mu-\varrho^\fq\rangle\beta-\langle\beta^\vee,w\varrho^\fp\rangle\beta-\beta+b\beta\\
=&\mu-(\langle\beta^\vee,\mu\langle w\rangle\rangle+1)\beta+b\beta\\
=&\mu-(n+kb)\beta+b\beta\\
=&\beta\downarrow\mu+mb\beta.
\end{align*}
Since $s_\beta(\beta)=-\beta$, the lemma follows.
\end{proof}

The family of standard type matrices is arranged into three subfamilies according to \cite{AA-diag-survey}: Cartan, super and the remainder. We next analyze them separately.

\subsubsection{Cartan type} This is the case in which all $i\in\I$ are Cartan vertexes and therefore all roots are Cartan roots. Its Weyl groupoid turns out be just the Weyl group of the Cartan matrix $C=C^\fq$, and hence it coincides with $\cW^\fq_{\car}$. The indecomposable matrices of Cartan type are listed in \cite[\S4]{AA-diag-survey}.

The previous results immediately imply the following.

\begin{corollary}\label{cor:cartan dot action}
Assume that $\fq$ is of Cartan type and $\pi:U_\fq^0\longrightarrow\bk$ satifies $\pi(K_j)=\pi(L_j)=1$ for all $j\in\I$. Then $[\mu]_{\link}\subset\cW^\fq_{\link}\bullet\mu$ for all $\mu\in\Z^\I$.
\qed
\end{corollary}

The natural example of Cartan type matrix is $\fq=(q^{d_ic_{ij}})_{i,j\in\I}$ as in Example \ref{ex:small qg}. In particular, if $\ord q$ is an odd prime, not $3$ if $\mathfrak{g}$ has a component of type $G_2$, and $\pi(K_i)=\pi(L_i)=1$, then the objects in the category $\cC^\fq_\bk$ associated to $u_q(\mathfrak{q})$ turn out to be $u_q(\mathfrak{q})$-modules of type $1$ in the sense of Lusztig, cf. \cite[\S2.4]{AJS}.

Let $\delta^\fq=\frac{1}{2}\sum_{\beta\in \Delta^\fq_+}\beta$ be the semi-sum of the positive roots of the Lie algebra associated to $C$. We can replace $\varrho^\fq$ with $\delta^\fq$ in the dot action under the next assumption. For instance, when $\fq$ is as in the above paragraph but this is not the case if $\ord q$ is even. In this way, we recover the usual dot action of the affine Weyl group.

\begin{lemma}\label{le:dot with delta cartan}
Suppose $b=b^\fq(\gamma)$ is constant for all $\gamma\in\Delta_+^\fq$. Let $\mu\in\Z^\I$ and $\beta\in\Delta_+^\fq$. Then $s_\beta\bullet\mu=s_\beta(\mu+\delta^\fq)-\delta^\fq+b\langle\beta^\vee,\delta^\fq\rangle\beta$.
\end{lemma}

\begin{proof}
From the proof of Corollary \ref{cor:cartan dot action}, we see that
\begin{align*}
s_\beta\bullet\mu=\mu-\langle\beta^\vee,\mu-\varrho^\fq\rangle\beta=\mu-\langle\beta^\vee,\mu+\delta^\fq\rangle\beta+b\langle\beta^\vee,\delta^\fq\rangle\beta
\end{align*}
as we wanted.
\end{proof}

\subsubsection{Super type}

These are the matrices $\fq$ whose root systems are isomorphic to the root systems of 
finite-dimensional contragredient Lie superalgebras
in characteristic 0 \cite{AA-diag-survey,AAYamane}. The indecomposable matrices of super type are listed in \cite[\S5]{AA-diag-survey}. An element in $\Delta_{\odd}^\fq:=\Delta^\fq\setminus\Delta_{\car}^\fq$ is called {\it odd root}. We can see by inspection on \cite[\S5]{AA-diag-survey} that $\fq(\alpha_i,\alpha_i)=-1$ for all simple odd roots $\alpha_i$. Then $q_\beta=-1$ for all $\beta\in\Delta_{\odd}^\fq$ and hence $\beta\downarrow\mu=\mu$ or $\mu-\beta$. For odd root, we can not always carry out $\downarrow$ as a dot action, see the example below. Instead we find that the classes $[\mu]_{\link}$ behave like in representation theory of Lie superalgebras, see for instance \cite{chenwang,panshu}. Let $\Z\Delta_{\odd}^\fq$ be the $\Z$-span of the odd roots in $\Z^\I$.

\begin{corollary}\label{cor:Linkage super}
Assume  that $\fq$ is of super type and $\pi:U_\fq^0\longrightarrow\bk$ satifies $\pi(K_j)=\pi(L_j)=1$ for all $j\in\I$. Then $[\mu]_{\link}\subset\cW^\fq_{\link}\bullet(\mu+\Z\Delta_{\odd}^\fq)$ for all $\mu\in\Z^\I$.
\end{corollary}

\begin{proof}
The set of Cartan roots is invariant by $s_\beta$ for all $\beta\in\Delta_{\car}^\fq$ by \cite[Lemma 3.6]{AARB} and hence so is $\Delta_{\odd}^\fq$. Then the lemma is a direct consequence of Lemma \ref{le:linkage via dot} and Theorem \ref{teo:strongly linked}.
\end{proof}

Let $\delta^\fq=\delta_{\car}^\fq-\delta_{\odd}^\fq$ with $\delta_{\car}^\fq=\frac{1}{2}\sum_{\beta\in \Delta^\fq_{+,\car}}\beta$ the semi-sum of the positive Cartan roots and $\delta_{\odd}^\fq=\frac{1}{2}\sum_{\beta\in \Delta^\fq_{+,\odd}}\beta$. The following is analogous to Lemma \ref{le:dot with delta cartan}.

\begin{lemma}
Suppose $b=b^\fq(\gamma)$ is constant for all $\gamma\in\Delta_{+,\car}^\fq$. Let $\mu\in\Z^\I$ and $\beta\in\Delta_{+,\car}^\fq$. Then $s_\beta\bullet\mu=s_\beta(\mu+\delta^\fq)-\delta^\fq+b\langle\beta^\vee,\delta_{\car}^\fq\rangle\beta$.
\qed
\end{lemma}

\begin{example}
Let $\fq$ and $\fp$ be as in Example \ref{ex:the example} and Example \ref{ex:the example groupoid}, respectively. Recall their root system in Example \ref{ex:the example root system}. Then $\pm\alpha_1$ and $\pm\alpha_2$ are odd roots and $\pm(\alpha_1+\alpha_2)=\sigma_1^\fp(\alpha_2)$ is a Cartan root because $2$ is Cartan vertex of $\fp$.

Let $\mu=\mu_1\alpha_1+\mu_2\alpha_2\in\Z^2$. Then
\begin{align*}
s_{\alpha_1}\bullet\mu=s_{\alpha_1}(\mu-\varrho^\fq)+\varrho^\fq=\mu-(2\mu_1-\mu_2-1)\alpha_1,
\end{align*}
and
\begin{align*}
n_{\alpha_1}^\pi(\mu)=\begin{cases}
                     1&\mbox{if }q^{\mu_2}=1,\\
                     0&\mbox{otherwise.}
                     \end{cases}
\end{align*}
Thus, $\alpha_1\downarrow\mu\neq s_{\alpha_1}\bullet\mu+mb^\fq(\alpha_1)\alpha_1$ for all $m\in\Z$ if $\ord q\neq\mu_2\in2\Z$ as $b^\fq(\alpha_1)=2$.
\end{example}

\subsubsection{} Besides the matrices of Cartan and super type there exists an infinite family of indecomposable matrices and one $2\times2$-matrix whose root systems are constant \cite[\S6]{AA-diag-survey}. For a matrix $\fq$ in this class and $\beta\in\Delta^\fq\setminus\Delta^\fq_{\car}$, the order of $q_\beta$ belongs to $\{2, 3,4\}$. Corollary \ref{cor:Linkage super} can be stated in this case but replacing $\Delta^\fq_{\odd}$ with $\Delta^\fq\setminus\Delta^\fq_{\car}$.

\subsection{Proof of Theorem \ref{teo:main teo} and Corollary \ref{cor:main cor}}

\begin{corollary}
Let $u_\fq$ be a small quantum group in the sense of Definition \ref{def:small quantum group}. {\it Mutatis mutandis}, all the results of this section hold for $u_\fq$ instead of $U_\fq$.
\end{corollary}

\begin{proof}
By definition there is a projection $U_\fq\longrightarrow u_\fq$ preserving the triangular decomposition as in \S\ref{subsec:quotients} and then all the results of this section can be restated for $u_\fq$ thanks to \eqref{eq:quotients}.  
\end{proof}

In particular, the above corollary applies to small quantum groups as in Figure \ref{fig:uq}, recall Example \ref{ex:uq}, and hence Theorem \ref{teo:main teo} and Corollary \ref{cor:main cor} follow. Alternatively, 
it is not difficult to see that the Lusztig isomorphisms descend to isomorphisms between these small quantum groups, and  $\widetilde{\mu}$ induces an algebra automorphism in $u_\fq^0$ for all $\mu\in\Z^\I$. Thus, one could repeat all the treatment of Sections \ref{sec:vermas} and \ref{sec:morphisms} for $u_\fq$, and give a direct proof of Theorem \ref{teo:main teo} and Corollary \ref{cor:main cor}.

\end{document}